\documentclass{daj}

\dajAUTHORdetails{%
  title = {Boolean Functions on $S_n$ Which Are Nearly Linear}, 
  author = {Yuval Filmus},
  plaintextauthor = {Yuval Filmus},
  plaintexttitle = {Boolean Functions on S_n Which Are Nearly Linear}, 
  keywords = {analysis of boolean functions, symmetric group},
}   

\dajEDITORdetails{%
   year={2021},
   number={25},
   received={19 July 2021},   
   published={13 December 2021},  
   doi={10.19086/da.30186},       
}   

\usepackage{amsmath,amssymb,amsfonts,amsthm,enumerate,hyperref,cleveref}
\usepackage{thmtools,thm-restate}
\usepackage{fullpage,color}

\providecommand{\RR}{\mathbb{R}}
\providecommand{\ZZ}{\mathbb{Z}}
\providecommand{\NN}{\mathbb{N}}
\DeclareMathOperator*{\E}{\mathbb{E}}
\DeclareMathOperator{\dist}{dist}
\DeclareMathOperator{\round}{round}

\providecommand{\cC}{\mathcal{C}}
\providecommand{\cL}{\mathcal{L}}
\providecommand{\va}{\mathbf{a}}
\providecommand{\vb}{\mathbf{b}}
\providecommand{\id}{\mathrm{id}}

\newtheorem{theorem}{Theorem}[section]
\newtheorem{lemma}[theorem]{Lemma}
\newtheorem{corollary}[theorem]{Corollary}
\newtheorem{conjecture}[theorem]{Conjecture}

\begin{document}

\begin{frontmatter}[classification=text]

\title{Boolean Functions on $S_n$ \\ Which Are Nearly Linear} 

\author[yuvalf]{Yuval Filmus\thanks{This project has received funding from the European Union's Horizon 2020 research and innovation programme under grant agreement No~802020-ERC-HARMONIC.}}


\begin{abstract}
We show that if $f\colon S_n \to \{0,1\}$ is $\epsilon$-close to linear in $L_2$ and $\E[f] \leq 1/2$ then $f$ is $O(\epsilon)$-close to a union of ``mostly disjoint'' cosets, and moreover this is sharp: any such union is close to linear. This constitutes a sharp Friedgut--Kalai--Naor theorem for the symmetric group.

Using similar techniques, we show that if $f\colon S_n \to \RR$ is linear, $\Pr[f \notin \{0,1\}] \leq \epsilon$, and $\Pr[f = 1] \leq 1/2$, then $f$ is $O(\epsilon)$-close to a union of mostly disjoint cosets, and this is also sharp; and that if $f\colon S_n \to \RR$ is linear and $\epsilon$-close to $\{0,1\}$ in $L_\infty$ then $f$ is $O(\epsilon)$-close in $L_\infty$ to a union of disjoint cosets.
\end{abstract}
\end{frontmatter}

\section{Introduction} \label{sec:introduction}

Consider a Boolean function $f\colon S_n \to \{0,1\}$, where $S_n$ is the symmetric group on $n$ elements, consisting of all permutations of $[n] = \{1,\ldots,n\}$. We say that $f$ is \emph{linear} if it can be written as
\[
 f = \sum_{i=1}^n \sum_{j=1}^n c_{i,j} x_{i,j},
\]
where $x_{i,j} \in \{0,1\}$ is the indicator for the input permutation sending $i$ to $j$. Ellis, Friedgut and Pilpel~\cite{EFP11} showed that if $f$ is a Boolean linear function on $S_n$ then $f$ must be a \emph{dictator}: either $f(\pi)$ depends only on $\pi(i)$ for some $i \in [n]$, or it depends only on $\pi^{-1}(j)$ for some $j \in [n]$ (see~\cite{DFLLV2021} for an alternative proof).

In this paper, we answer several relaxed versions of the same question:
\begin{enumerate}
 \item What is the structure of a Boolean function $f\colon S_n \to \{0,1\}$ which is $\epsilon$-close to linear, in the sense that there exists a linear function $g\colon S_n \to \RR$ satisfying $\E[(f-g)^2] \leq \epsilon$?
 \item What is the structure of a linear function $f\colon S_n \to \RR$ which is $\epsilon$-close to Boolean, in the sense that $\E[\dist(f,\{0,1\})^2] \leq \epsilon$? (Here $\dist(f,\{0,1\}) = \min(|f-0|, |f-1|)$.)
 \item What is the structure of a linear function $f\colon S_n \to \RR$ which is $\epsilon$-close to Boolean, in the sense that $\Pr[f \notin \{0,1\}] \leq \epsilon$?
 \item What is the structure of a linear function $f\colon S_n \to \RR$ which is $\epsilon$-close to Boolean, in the sense that $\dist(f(\pi), \{0,1\}) \leq \epsilon$ for all $\pi \in S_n$?
\end{enumerate}
The first two questions measure distance in the $L_2$~metric, the third one measures distance in the $L_0$~metric, and the fourth one measures distance in the $L_\infty$~metric. Below, all notions of distance are $L_2$ unless explicitly stated.

The answer to the first three questions is similar, and involves the notion of an $\epsilon$-disjoint family of cosets. A \emph{coset} $(i,j)$ consists of all permutations sending $i$ to $j$, for some $i,j \in [n]$.
Two cosets $(i,j),(k,\ell)$ are \emph{disjoint} if the corresponding sets of permutations are disjoint; this happens when either $i = k$ or $j = \ell$ (but not both).
An \emph{$\epsilon$-disjoint family of cosets} is a collection $\cC$ of $O(n)$ cosets such that the number of pairs $(i,j),(k,\ell) \in \cC$ which are not disjoint is $O(\epsilon n^2)$.

We can now state our main structure theorems, answering all of the questions above. In the first three theorems, we assume that either $\E[f] \leq 1/2$ or $\Pr[f=1] \leq 1/2$. To obtain a theorem covering all possible $f$, we simply need to consider $1-f$ when $\E[f] > 1/2$ or $\Pr[f=1] > 1/2$.

\begin{restatable}{theorem}{mainone} \label{thm:main-1}
 If $f\colon S_n \to \{0,1\}$ is $\epsilon$-close to linear and $\E[f] \leq 1/2$ then there exists an $\epsilon$-disjoint family of cosets $\cC$ such that $f$ is $\epsilon$-close to $g = \max_{(i,j) \in \cC} x_{i,j}$, that is, $\Pr[f \neq g] = O(\epsilon)$.
 
 Conversely, if $\cC$ is an $\epsilon$-disjoint family of cosets then $g = \max_{(i,j) \in \cC} x_{i,j}$ is $O(\epsilon)$-close to linear.
\end{restatable}

\begin{restatable}{theorem}{maintwo} \label{thm:main-2}
 If $f\colon S_n \to \RR$ is linear, $\epsilon$-close to Boolean, and satisfies $\E[f] \leq 1/2$ then there exists an $\epsilon$-disjoint family of cosets $\cC$ such that $f$ is $O(\epsilon)$-close to $g = \sum_{(i,j) \in \cC} x_{i,j}$, that is, $\E[(f-g)^2] = O(\epsilon)$.
 
 Conversely, if $\cC$ is an $\epsilon$-disjoint family of cosets then $g = \sum_{(i,j) \in \cC} x_{i,j}$ is $O(\epsilon)$-close to Boolean.
\end{restatable}

\begin{restatable}{theorem}{mainthree} \label{thm:main-3}
 If $f\colon S_n \to \RR$ is linear, $\epsilon$-close to Boolean in~$L_0$, and satisfies $\Pr[f=1] \leq 1/2$ then there exists an $\epsilon$-disjoint family of cosets $\cC$ such that $f$ is $O(\epsilon)$-close to $g = \sum_{(i,j) \in \cC} x_{i,j}$, that is, $\Pr[f \neq g] = O(\epsilon)$.
 
 Conversely, if $\cC$ is an $\epsilon$-disjoint family of cosets then $g = \sum_{(i,j) \in \cC} x_{i,j}$ is $O(\epsilon)$-close to Boolean in $L_0$.
\end{restatable}

Answering the fourth question, we can guarantee that the cosets in $\cC$ are completely disjoint.

\begin{restatable}{theorem}{mainfour} \label{thm:main-4}
There exists a constant $\epsilon_0 > 0$ for which the following holds.

If $f\colon S_n \to \RR$ is linear and $\epsilon$-close to Boolean in~$L_\infty$, where $\epsilon \leq \epsilon_0$, then $g(\pi) = \round(f(\pi), \{0,1\})$ is a dictator.
\end{restatable}

Since $x_{i,j}$ is the indicator function of the coset $(i,j)$, the function $g$ in \Cref{thm:main-1} is the indicator function of the union of the cosets in $\cC$.

\smallskip

If $f$ is not very close to a constant, then we can approximate $f$ by a dictator in all settings.

\begin{restatable}{theorem}{mainonea} \label{thm:main-1a}
 The following holds for some constant $K>0$.
 If $f\colon S_n \to \{0,1\}$ is $\epsilon$-close to linear and $\delta \leq \E[f] \leq 1-\delta$, where $\delta \geq K\sqrt{\epsilon}$, then $f$ is $O(\epsilon/\delta)$-close to a dictator.
\end{restatable}

\begin{restatable}{theorem}{maintwoa} \label{thm:main-2a}
 The following holds for some constant $K>0$.
 If $f\colon S_n \to \RR$ is linear, $\epsilon$-close to Boolean, and $\delta \leq \E[f] \leq 1-\delta$, where $\delta \geq K\sqrt{\epsilon}$, then $\Pr[f \neq g] = O(\epsilon/\delta)$ for some dictator $g$.
\end{restatable}

\begin{restatable}{theorem}{mainthreea} \label{thm:main-3a}
 The following holds for some constant $K>0$. 
 If $f\colon S_n \to \RR$ is linear, $\epsilon$-close to Boolean in $L_0$, and $\delta \leq \Pr[f=1] \leq 1-\delta$, where $\delta \geq K\sqrt{\epsilon}$, then $\Pr[f \neq g] = O(\epsilon/\delta)$ for some dictator $g$.
\end{restatable}

The tradeoff between $\delta$ and $\epsilon/\delta$ is tight, as the following example shows:
\[
 f = \sum_{j=1}^{\delta n} x_{1j} + \sum_{j=1}^{(\epsilon/\delta)n} x_{2j}.
\]
This function is roughly $(\delta + \epsilon/\delta)$-far from constant and $\min(\delta,\epsilon/\delta)$-far from a dictator.

\Cref{thm:main-1,thm:main-2,thm:main-1a,thm:main-2a} improve on prior work of the author together with Ellis and Friedgut~\cite{EFF1,EFF2}, on which we shall have more to say below. (We are not aware of prior work considering the~$L_0$ or $L_\infty$ versions.)

\paragraph{Background} The $n$-dimensional \emph{Boolean cube} is the set $\{0,1\}^n$. A function $f$ on the Boolean cube is \emph{linear} if it can be written as
\[
 f = c + \sum_{i=1}^n c_i x_i,
\]
where $x_i$ is the value of the $i$'th input coordinate. It is not hard to see that a Boolean linear function on the Boolean cube is either constant or of the form $x_i$ or $1 - x_i$.

The classical Friedgut--Kalai--Naor (FKN) theorem~\cite{FKN02} states that if $f$ is a Boolean function which is $\epsilon$-close to linear, then $f$ is $O(\epsilon)$-close to some Boolean linear function, and the same holds if we assume that $f$ is a linear function which is $\epsilon$-close to Boolean. (The $L_0$~version is less interesting, since if $f$ is a linear function then either $f$ is Boolean or $\Pr[f \notin \{0,1\}] \geq 1/4$; and the $L_\infty$~version is easy.)

The FKN theorem is a standard result in an area of study known as \emph{Boolean function analysis}~\cite{ODonnell}. Traditionally, Boolean function analysis concerns itself with functions on the Boolean cube (and sometimes, Gaussian space). In the past decade, Boolean function analysis has been expanded to many other domains, due to applications in combinatorics and theoretical computer science. 

One of the first domains to be explored in this way is the symmetric group. As part of their proof of the Deza--Frankl conjecture in Erd\H{o}s--Ko--Rado theory, Ellis, Friedgut and Pilpel~\cite{EFP11} showed that a Boolean linear function on $S_n$ is a dictator; we call this an \emph{exact} FKN theorem. Motivated by this application, Ellis, Friedgut and the author studied the structure of Boolean functions on $S_n$ which are close to linear, in a sequence of two papers~\cite{EFF1,EFF2}. The first paper shows that if $f\colon S_n \to \{0,1\}$ is sparse, that is $\E[f] = c/n$, and $\epsilon c/n$-close to linear, then $f$ is $O(c^2 (\sqrt{\epsilon} + 1/n)/n))$-close to a maximum of $\round(c)$ cosets, and moreover $c$ is close to an integer. The second paper shows that if $f\colon S_n \to \{0,1\}$ is balanced, that is $\eta \leq \E[f] \leq 1-\eta$, and $\epsilon$-close to linear, then $f$ is $O(\eta^{-1} (\epsilon^{1/7} + 1/n^{1/3}))$-close to a dictator. Both of these results are improved by our main theorems.

Exact and general FKN theorems have been proved for many other domains, such as product domains~\cite{RS15,JOW15,Nayar14}, the slice~\cite{Filmus2016b}, the multislice~\cite{Filmus2020}, and simplicial complexes~\cite{DDFH2018}. An exact FKN theorem was proved for several classical association schemes~\cite{FI2019a}, but general FKN theorems are not known in these domains.

\smallskip

The higher-dimensional analog of the exact FKN theorem states that a degree~$d$ Boolean function on the cube depends on $O(2^d)$ coordinates~\cite{NisanSzegedy94,CHS20,Wellens19}, and can be expressed as a decision tree of depth $O(d^3)$~\cite{NisanSzegedy94,Midrijanis04}. The corresponding FKN theorem is due to Kindler and Safra~\cite{Kindler02,KS04}, and states that if a Boolean function is $\epsilon$-close to degree~$d$ then it is $O(\epsilon)$-close to a Boolean degree~$d$ function. Together with Ihringer, we extended the former result to the slice~\cite{FI2019b}, and the latter result was extended to the slice by Keller and Klein~\cite{KK20}. Another extension of the Kindler--Safra theorem, to the biased Boolean cube, is due to Dinur, Harsha and the author~\cite{DFH2019}; this domain shares some of the confounding features of the symmetric group.

A function on the symmetric group has degree~$d$ if it can be written as a degree~$d$ polynomial in the variables $x_{i,j}$. When $d \geq 2$, a Boolean degree~$d$ function no longer depends on a constant number of ``coordinates'', but it can be written as a decision tree of depth $O(d^8)$~\cite{DFLLV2021}, whose internal nodes contain queries of the form ``$\pi(i) = ?$'' or ``$\pi^{-1}(j) = ?$''.

It is not clear what the extension of the Kindler--Safra theorem to the slice should look like; some ideas can be gleaned from~\cite{DFH2019}. Nevertheless, in the balanced case it is natural to conjecture the following.

\begin{conjecture} \label{cnj:kindler-safra-sn}
 If $f\colon S_n \to \{0,1\}$ is $\epsilon$-close to degree~$d$ and $1/3 \leq \E[f] \leq 2/3$ then $f$ is $O(\epsilon)$-close to a Boolean degree~$d$ function.
\end{conjecture}

\paragraph{Structure of the paper} We prove \Cref{thm:main-1,thm:main-2,thm:main-1a,thm:main-2a} in \Cref{sec:L2}, \Cref{thm:main-3,thm:main-3a} in \Cref{sec:L0}, and \Cref{thm:main-4} in \Cref{sec:Linf}.

\paragraph{Notation} We assume familiarity with Boolean function analysis on the Boolean cube~\cite{ODonnell}. 
For a finite set $S$, we define $\dist(x,S) = \min_{y \in S}(|x-y|)$ and $\round(x,S) = \operatorname{argmin}_{y \in S}(|x-y|)$ (we will also allow $S = \ZZ$). We will frequently use the inequality $(a+b)^2 \leq 2a^2 + 2b^2$. For a predicate $P$, $[P] = 1$ if $P$ holds and $[P] = 0$ otherwise.

\section{Approximation in \texorpdfstring{$L_2$}{L2}} \label{sec:L2}

In this section we prove \Cref{thm:main-2,thm:main-2a}, and then derive \Cref{thm:main-1,thm:main-1a}. 

The proof of \Cref{thm:main-2,thm:main-2a} proceeds as follows. The starting point is a function $f\colon S_n \to \RR$ which satisfies $\E[\dist(f,\{0,1\})^2] = \epsilon$. In \Cref{sec:L2-sparse}, we show that $f$ can be approximated by a function of the form
\[
 g = e + \sum_{i,j} e_{i,j} x_{i,j},
\]
where $e$ is an integer, $e_{i,j} \in \{0,\pm 1\}$, and $O(n)$ many of the $e_{i,j}$ are non-zero. This uses a reduction to the Boolean cube described in \Cref{sec:L2-reduction}, followed by an application of the classical FKN theorem on the Boolean cube, whose statement appears in \Cref{sec:L2-FKN}.

The next step, appearing in \Cref{sec:sporadic}, is to find an alternative representation
\[
 g = r + \sum_{i,j} r_{i,j} x_{i,j},
\]
where $r$ is an integer, the $r_{i,j}$ are bounded integers, $O(n)$ many of the $r_{i,j}$ are non-zero, and furthermore, with probability $\Omega(1)$, a random permutation $\pi$ satisfies $r_{i,\pi(i)} = 0$ for all $i \in [n]$.

Given this, we deduce in \Cref{sec:sum-of-cosets} that either $g$ or $1-g$ is close to a function of the form
\[
 h = \sum_{(i,j) \in \cC} x_{i,j},
\]
where $\cC$ is an $\epsilon$-disjoint family of cosets. This proves \Cref{thm:main-2} up to the possibility that $h$ approximates $1-g$ rather than $g$.

We prove a version of \Cref{thm:main-2a} for $h$ in \Cref{sec:constant-or-dictator}, and use it to derive \Cref{thm:main-2,thm:main-2a} in \Cref{sec:L2-main}, where we also deduce \Cref{thm:main-1,thm:main-1a}.

\subsection{FKN theorem on the cube} \label{sec:L2-FKN}

We start by stating the classical Friedgut--Kalai--Naor theorem on the Boolean cube.

\begin{theorem}[\cite{FKN02}] \label{thm:FKN}
 Let $f\colon \{0,1\}^n \to \{0,1\}$ be a Boolean function satisfying $\E[(f^{>1})^2] = \epsilon$. Then $\E[(f - g)^2] = O(\epsilon)$ for some function $g \in \{0, 1, x_1, \ldots, x_n, 1-x_1, \ldots, 1-x_n\}$.
\end{theorem}

This formulation of the FKN theorem states that a Boolean function which is close to linear is close to a Boolean function depending on at most one coordinate. Another formulation states that a linear function which is close to Boolean is close to a Boolean function depending on at most one coordinate.

\begin{corollary} \label{cor:FKN}
 Let $f\colon \{0,1\}^n \to \RR$ be given by
\[
 f(x_1,\ldots,x_n) = c + \sum_{i=1}^n c_i x_i.
\]
 If $\E[\dist(f,\{0,1\})^2] = \epsilon$ then there exist $d_1,\ldots,d_n \in \{0,\pm 1\}$ such that
\[
 \sum_{i=1}^n (c_i - d_i)^2 = O(\epsilon).
\]
 Furthermore, at most one of the $d_i$ is non-zero.
\end{corollary}
\begin{proof}
 Let $F = \round(f,\{0,1\})$. Then $\E[(F^{>1})^2] \leq \E[(F-f)^2] = \E[\dist(f,\{0,1\})^2] = \epsilon$. According to \Cref{thm:FKN}, $\E[(F - g)^2] = O(\epsilon)$ for some function $g$ which is of one of the forms $0,1,x_i,1-x_i$, where $i \in [n]$. Thus $\E[(f - g)^2] = O(\E[(f - F)^2] + \E[(F - g)^2]) = O(\epsilon)$.
 
 The next step is to consider the Fourier expansions of $f$ and $g$, and to apply Parseval's identity $\|f - g\|^2 = \sum_S (\hat{f}(S) - \hat{g}(S))^2$. The Fourier expansion of $f$ is
\[
 f = c + \frac{1}{2} \sum_{i=1}^n c_i - \frac{1}{2} \sum_{i=1}^n c_i (-1)^{x_i}.
\]
 The Fourier expansion of $g$ is one of the following:
\[
 0, 1, \frac{1}{2} - \frac{1}{2} (-1)^{x_j}, \frac{1}{2} + \frac{1}{2} (-1)^{x_j}.
\]
 In all cases, we can write
\[
 g = d - \frac{1}{2} \sum_{i=1}^n d_i (-1)^{x_i},
\]
 where $d_1,\ldots,d_n \in \{0,\pm1\}$, and at most one $d_i$ is non-zero. Parseval's identity implies that
\[
 \frac{1}{4} \sum_{i=1}^n (c_i - d_i)^2 = \sum_{|S| = 1} (\hat{f}(S) - \hat{g}(S))^2 \leq \E[(f - g)^2] = O(\epsilon),
\]
 from which the corollary immediately follows.
\end{proof}

\subsection{Reduction to the cube} \label{sec:L2-reduction}

The starting point of the proof is a reduction to the Boolean cube. Recall that our initial goal is proving \Cref{thm:main-2,thm:main-2a}. Accordingly, we are given a function $f\colon S_n \to \RR$ of the form
\[
 f = \sum_{i=1}^n \sum_{j=1}^n c_{i,j} x_{i,j}
\]
such that $\E[\dist(f,\{0,1\})^2] = \epsilon$.

Let $\va = (a_1,\ldots,a_n)$ and $\vb = (b_1,\ldots,b_n)$ be two permutations of $[n]$, and let $S_{\va,\vb} \subset S_n$ consist of all permutations which send $\{a_{2t-1},a_{2t}\}$ to $\{b_{2t-1},b_{2t}\}$ for $t \in [\lfloor n/2 \rfloor]$; if $n$ is odd, then these permutations necessarily send $a_n$ to $b_n$. Let $f_{\va,\vb}$ be the restriction of $f$ to $S_{\va,\vb}$, and let $\epsilon_{\va,\vb} = \E[\dist(f_{\va,\vb},\{0,1\})^2]$.

We identify $S_{\va,\vb}$ with an $\lfloor n/2 \rfloor$-dimensional cube $\{0,1\}^{\lfloor n/2 \rfloor}$ in the following way. Given $(x_1,\ldots,x_{\lfloor n/2 \rfloor}) \in \{0,1\}^{\lfloor n/2 \rfloor}$, if $x_t = 1$ then the permutation sends $a_{2t-1}$ to $b_{2t-1}$ and $a_{2t}$ to $b_{2t}$, and if $x_t = 0$ then the permutation sends $a_{2t-1}$ to $b_{2t}$ and $a_{2t}$ to $b_{2t-1}$. Thus
\[
 f_{\va,\vb}(x_1,\ldots,x_{\lfloor n/2 \rfloor}) = c + \sum_{t=1}^{\lfloor n/2 \rfloor} (c_{a_{2t-1},b_{2t-1}} + c_{a_{2t},b_{2t}} - c_{a_{2t-1},b_{2t}} - c_{a_{2t},b_{2t-1}}) x_t,
\]
where $c$ is given by slightly different formulas according to the parity of $n$: when $n$ is even,
\[
 c = \sum_{t=1}^{n/2} (c_{a_{2t-1},b_{2t}} + c_{a_{2t},b_{2t-1}}),
\]
and when $n$ is odd,
\[
 c = \sum_{t=1}^{(n-1)/2} (c_{a_{2t-1},b_{2t}} + c_{a_{2t},b_{2t-1}}) + c_{a_n,b_n}.
\]

If we choose $\va,\vb$ uniformly at random and then a random permutation in $S_{\va,\vb}$, then we obtain a uniformly random permutation. This shows that
\[
 \E_{\va,\vb}[\epsilon_{\va,\vb}] = \epsilon.
\]
In the following subsection we use this, together with \Cref{cor:FKN}, to approximate $f$ by a sparse weighted sum of $x_{i,j}$'s.

\subsection{Sparse representation} \label{sec:L2-sparse}

Our goal in this subsection is to prove the following lemma.

\begin{lemma} \label{lem:L2-sparse}
 Let $f\colon S_n \to \RR$ be a linear function satisfying $\E[\dist(f,\{0,1\})^2] = \epsilon$, where $\epsilon \leq 1$. Then there is an integer $e$ and integers $e_{i,j} \in \{0,\pm1\}$ such that the function
 \[
  g = e + \sum_{i=1}^n \sum_{j=1}^n e_{i,j} x_{i,j}
 \]
 satisfies $\E[(f-g)^2] = O(\epsilon)$. Moreover, only $O(n)$ many of the $e_{i,j}$ are non-zero.
 
 Furthermore, $\Pr[g \notin \{0,1\}] = O(\epsilon)$.
\end{lemma}

Since $f$ is linear, it can be written as
\[
 f = \sum_{i=1}^n \sum_{j=1}^n c_{i,j} x_{i,j}.
\]

If $n = 1$ then $f = c_{1,1}$ is constant, where $\dist(c_{1,1},\{0,1\})^2 = \epsilon$. If $|c_{1,1}|^2 = \epsilon$ then we can take $g = 0$, and if $|c_{1,1} - 1|^2 = \epsilon$ then we can take $g = 1$. From now on, we assume that $n \geq 2$.

Let $\va,\vb \in S_n$. Applying \Cref{cor:FKN} to $f_{\va,\vb}$ shows that for some $d_1,\ldots,d_{\lfloor n/2 \rfloor} \in \{0,\pm 1\}$, we have
\[
 \sum_{t=1}^{\lfloor n/2 \rfloor} (c_{a_{2t-1},b_{2t-1}} + c_{a_{2t},b_{2t}} - c_{a_{2t-1},b_{2t}} - c_{a_{2t},b_{2t-1}} - d_t)^2 = O(\epsilon_{\va,\vb}).
\]
Moreover, at most one of $d_1,\ldots,d_{\lfloor n/2 \rfloor}$ is non-zero. In order to capitalize on that, define
\[
 e_{i_1,i_2,j_1,j_2} = \round(c_{i_1,j_1} + c_{i_2,j_2} - c_{i_1,j_2} - c_{i_2,j_1}, \{0,\pm 1\}).
\]
Clearly
\begin{multline} \label{eq:sparse-1}
 \sum_{t=1}^{\lfloor n/2 \rfloor} (c_{a_{2t-1},b_{2t-1}} + c_{a_{2t},b_{2t}} - c_{a_{2t-1},b_{2t}} - c_{a_{2t},b_{2t-1}} - e_{a_{2t-1},a_{2t},b_{2t-1},b_{2t}})^2 \leq \\
 \sum_{t=1}^{\lfloor n/2 \rfloor} (c_{a_{2t-1},b_{2t-1}} + c_{a_{2t},b_{2t}} - c_{a_{2t-1},b_{2t}} - c_{a_{2t},b_{2t-1}} - d_t)^2 = O(\epsilon_{\va,\vb}),
\end{multline}
which implies that
\[
 \sum_{t=1}^{\lfloor n/2 \rfloor} (e_{a_{2t-1},a_{2t},b_{2t-1},b_{2t}} - d_t)^2 = O(\epsilon_{\va,\vb}).
\]
Since at most one of the $d_t$ is non-zero, this implies that
\begin{equation} \label{eq:sparse-2}
 \sum_{t=1}^{\lfloor n/2 \rfloor} [e_{a_{2t-1},a_{2t},b_{2t-1},b_{2t}} \neq 0] \leq 1 + O(\epsilon_{\va,\vb}).
\end{equation}

Choosing $\va,\vb$ uniformly at random and taking expectation of~\eqref{eq:sparse-1}, we obtain
\[
 \lfloor n/2 \rfloor \E_{\substack{i_1,i_2,j_1,j_2 \in [n] \\ i_1 \neq i_2, j_1 \neq j_2}}[(c_{i_1,j_1} + c_{i_2,j_2} - c_{i_1,j_2} - c_{i_2,j_1} - e_{i_1,i_2,j_1,j_2})^2] = O\bigl(\E[\epsilon_{\va,\vb}]\bigr) = O(\epsilon).
\]
Doing the same for~\eqref{eq:sparse-2}, we obtain
\[
 \lfloor n/2 \rfloor \Pr_{\substack{i_1,i_2,j_1,j_2 \in [n] \\ i_1 \neq i_2, j_1 \neq j_2}}[e_{i_1,i_2,j_1,j_2} \neq 0] \leq 1 + O(\epsilon) = O(1).
\]
Markov's inequality shows that each of the following holds with probability at least $2/3$ over a random choice of $i_1,j_1 \in [n]$:
\begin{gather}
 \E_{\substack{i_2,j_2 \in [n] \\ i_2 \neq i_1, j_2 \neq j_1}}[(c_{i_1,j_1} + c_{i_2,j_2} - c_{i_1,j_2} - c_{i_2,j_1} - e_{i_1,i_2,j_1,j_2})^2] = O(\epsilon/n), \label{eq:sparse-1a} \\
 \Pr_{\substack{i_2,j_2 \in [n] \\ i_2 \neq i_1, j_2 \neq j_1}}[e_{i_1,i_2,j_1,j_2} \neq 0] = O(1/n). \label{eq:sparse-2a}
\end{gather}
Hence there exists a choice of $i_1,j_1 \in [n]$ for which both of these statements hold simultaneously.

\smallskip

Let us take stock of our situation. We have shown that on average,
\[
 d_{i_2,j_2} := c_{i_1,j_1} + c_{i_2,j_2} - c_{i_1,j_2} - c_{i_2,j_1} \approx e_{i_1,i_2,j_1,j_2},
\]
where $e_{i_1,i_2,j_1,j_2} \in \{0,\pm 1\}$, and only $O(n)$ of these coefficients are non-zero.
This suggests finding a different representation of $f$ which involves the left-hand sides. Indeed,
\begin{multline*}
 \sum_{i=1}^n \sum_{j=1}^n d_{i,j} x_{i,j} =
 \sum_{i=1}^n \sum_{j=1}^n c_{i_1,j_1} x_{i,j} - \sum_{i=1}^n \sum_{j=1}^n c_{i_1,j} x_{i,j} - \sum_{i=1}^n \sum_{j=1}^n c_{i,j_1} x_{i,j} + \sum_{i=1}^n \sum_{j=1}^n c_{i,j} x_{i,j} = \\
 n c_{i_1,j_1} - \sum_{j=1}^n c_{i_1,j} - \sum_{i=1}^n c_{i,j_1} + \sum_{i=1}^n \sum_{j=1}^n c_{i,j} x_{i,j},
\end{multline*}
which implies that for an appropriate $d$,
\[
 f = d + \sum_{i=1}^n \sum_{j=1}^n d_{i,j} x_{i,j}.
\]
Furthermore, defining $e_{i,j} := e_{i_1,i,j_1,j}$, \eqref{eq:sparse-1a} implies that
\[
 \sum_{i=1}^n \sum_{j=1}^n (d_{i,j} - e_{i,j})^2 = O(n\epsilon),
\]
since $d_{i,j} = e_{i,j} = 0$ if $i = i_1$ or $j = j_1$. Moreover, \eqref{eq:sparse-2a} shows that only $O(n)$ of the coefficients $e_{i,j}$ are non-zero.

We would now like to say that if we replace the coefficients $d_{i,j}$ with the coefficients $e_{i,j}$ then the resulting function is close to $f$. While this is correct, we have to be careful when doing the switch, replacing $x_{i,j}$ with the corresponding centered term $x_{i,j} - 1/n$. For an appropriate $d'$,
\[
 f = d' + \sum_{i=1}^n \sum_{j=1}^n d_{i,j} (x_{i,j} - 1/n).
\]
Define
\[
 g = d' + \sum_{i=1}^n \sum_{j=1}^n e_{i,j} (x_{i,j} - 1/n),
\]
which for an appropriate $e$ can also be written as
\[
 g = e + \sum_{i=1}^n \sum_{j=1}^n e_{i,j} x_{i,j}.
\]
Below, we will show that $\E[(f - g)^2] = O(\epsilon)$. In order to complete the proof of \Cref{lem:L2-sparse}, we show that we can slightly modify $g$ so that $e$ becomes an integer.

Since $\E[\dist(f,\{0,1\})^2] = \epsilon$ and $\E[(f - g)^2] = O(\epsilon)$, we have $\E[\dist(g,\{0,1\})^2] = O(\epsilon)$. Since all $e_{i,j}$ are integers, this implies that $\dist(e,\ZZ)^2 = O(\epsilon)$. Let $e'$ be the integer closest to $e$, and define
\[
 g' = e' + \sum_{i=1}^n \sum_{j=1}^n e_{i,j} x_{i,j}.
\]
Then $\E[(g-g')^2] = (e-e')^2 = O(\epsilon)$ and so $\E[(f-g')^2] = \epsilon$.

The function $g'$ satisfies $\E[\dist(g',\{0,1\})^2] = O(\E[\dist(f,\{0,1\})^2] + \E[(f-g')^2]) = O(\epsilon)$. Since $g'$ is integer-valued, this implies furthermore that $\Pr[g' \notin \{0,1\}] = O(\epsilon)$.

\smallskip

In the rest of this subsection, we show that $\E[(f - g)^2] = O(\epsilon)$. If we define $\delta_{i,j} = d_{i,j} - e_{i,j}$ then on the one hand
\begin{equation} \label{eq:sparse-1c}
 \frac{1}{n} \sum_{i=1}^n \sum_{j=1}^n \delta_{i,j}^2 = O(\epsilon),
\end{equation}
and on the other hand
\[
 f - g = \sum_{i=1}^n \sum_{j=1}^n \delta_{i,j} (x_{i,j} - 1/n).
\]

Calculation shows that
\[
 \E[(x_{i_1,j_1} - 1/n) (x_{i_2,j_2} - 1/n)] = \E[x_{i_1,j_1} x_{i_2,j_2}] - \frac{1}{n^2} = 
 \begin{cases}
  \frac{n-1}{n^2} & \text{if } i_1 = j_1 \text{ and } i_2 = j_2, \\
  -\frac{1}{n^2} & \text{if } i_1 = j_1 \text{ or } i_2 = j_2 \text{ but not both}, \\
  \frac{1}{n^2(n-1)} & \text{if } i_1 \neq j_1 \text{ and } i_2 \neq j_2.
 \end{cases}
\]
It follows that
\begin{multline*}
 \E[(f-g)^2] =
 \underbrace{\frac{n-1}{n^2} \sum_{i=1}^n \sum_{j=1}^n \delta_{i,j}^2}_{A} -
 \underbrace{\frac{1}{n^2} \sum_{i=1}^n \sum_{\substack{j_1,j_2 \in [n] \\ j_1 \neq j_2}} \delta_{i,j_1} \delta_{i,j_2}}_{B} -
 \underbrace{\frac{1}{n^2} \sum_{j=1}^n \sum_{\substack{i_1,i_2 \in [n] \\ i_1 \neq i_2}} \delta_{i_1,j} \delta_{i_2,j}}_{C} + \\
 \underbrace{\frac{1}{n^2(n-1)} \sum_{\substack{i_1,i_2 \in [n] \\ i_1 \neq i_2}} \sum_{\substack{j_1,j_2 \in [n] \\ j_1 \neq j_2}} \delta_{i_1,j_1} \delta_{i_2,j_2}}_{D}.
\end{multline*}
\Cref{eq:sparse-1c} immediately implies that $A = O(\epsilon)$. As for $B$, using $|\delta_1 \delta_2| \leq (\delta_1^2 + \delta_2^2)/2$, we can bound
\[
 |B| \leq \frac{n-1}{n^2} \sum_{i=1}^n \sum_{j=1}^n \delta_{i,j}^2 = O(\epsilon).
\]
The same bound holds for $C$. Finally, we can bound $D$ in the same way:
\[
 |D| \leq \frac{n-1}{n^2} \sum_{i=1}^n \sum_{j=1}^n \delta_{i,j}^2 = O(\epsilon).
\]
Altogether, this shows that $\E[(f-g)^2] = O(\epsilon)$, as wanted.

\subsection{Sporadic representation} \label{sec:sporadic}

\Cref{lem:L2-sparse} constructs a function $g$ which is almost Boolean in an $L_0$~sense. Furthermore, $g$ has a representation $g = e + \sum_{i,j} e_{i,j} x_{i,j}$ which is sparse in the sense that the \emph{support} $\{(i,j) : e_{i,j} \neq 0\}$ has size $O(n)$. In order to make further progress, we need to find an alternative representation of $g$ whose support satisfies the stronger property that with constant probability, a random permutation hits \emph{no} coset in the support. We will have to pay for this by allowing larger coefficients in the sum.

\begin{lemma} \label{lem:sporadic}
Let $g\colon S_n \to \ZZ$ be a function of the form
\[
 g = e + \sum_{i=1}^n \sum_{j=1}^n e_{i,j} x_{i,j},
\]
where $e \in \ZZ$; $e_{i,j} \in \{0,\pm 1\}$ for all $i,j \in [n]$; at most $Cn$ of the $e_{i,j}$ are non-zero; and $\Pr[g \notin \{0,1\}] = \epsilon$. There exist constants $b,N \in \ZZ$ and $\gamma > 0$ (possibly depending on $C$) such that if $n \geq N$ then $g$ has an alternative representation
\[
 g = r + \sum_{i=1}^n \sum_{j=1}^n r_{i,j} x_{i,j},
\]
where $r \in \ZZ$; $r_{i,j} \in \{-b, \ldots, b\}$ for all $i,j \in [n]$; at most $O_C(n)$ of the $r_{i,j}$ are non-zero; and
\[
 \Pr_{\pi \in S_n}[r_{i,\pi(i)} = 0 \text{ for all } i \in [n]] \geq \gamma.
\]
\end{lemma}

The first step in the proof is constructing the new representation. An obstacle for the ``hitting'' property is a \emph{line} (row or column) in which all or almost all of the $e_{i,j}$ are non-zero. Accordingly, let $\alpha_i \in \{-1,0,1\}$ be a most common value of $e_{i,1},\ldots,e_{i,n}$, let $\beta_j \in \{-1,0,1\}$ be the most common value of $e_{1,j},\ldots,e_{n,j}$, and define $r_{i,j} = e_{i,j} - \alpha_i - \beta_j$, so that
\[
 g = e + \sum_{i=1}^n \sum_{j=1}^n (r_{i,j} + \alpha_i + \beta_j) x_{i,j} = e + \sum_{i=1}^n \alpha_i + \sum_{j=1}^n \beta_j + \sum_{i=1}^n \sum_{j=1}^n r_{i,j} x_{i,j}.
\]
Note that $r = e + \sum_i \alpha_i + \sum_j \beta_j \in \ZZ$, and $r_{i,j} \in \{0,\pm1,\pm2,\pm3\}$.

We claim that most of the coefficients $\alpha_i,\beta_j$ are equal to zero. Indeed, if $\alpha_i \neq 0$ then among $e_{i,1},\ldots,e_{i,n}$, at least $(2/3)n$ are non-zero. Therefore at most $Cn/(2/3)n = (3/2)C$ of the $\alpha_i$ are non-zero. Similarly, at most $(3/2)C$ of the $\beta_j$ are non-zero.
If $r_{i,j} \neq 0$ then at least one of $e_{i,j},\alpha_i,\beta_j$ is non-zero. This shows that the number of non-zero $r_{i,j}$ is at most $Cn + (3/2)Cn + (3/2)Cn = 4Cn$.

It remains to prove the hitting property: if $R = \{(i,j) \in [n]^2 : r_{i,j} \neq 0\}$, then with probability $\Omega(1)$, a random permutation $\pi$ ``avoids'' $R$, in the sense that $(i,\pi(i)) \notin R$ for all $i \in [n]$. We will prove this using two properties of $R$. First, as noted above, $|R| \leq 4Cn$. Second, we can bound the intersection of $R$ with any row or column.
Given $i \in [n]$, notice that by construction, $e_{i,1} - \alpha_i, \ldots, e_{i,n} - \alpha_i$ contains at most $(2/3)n$ many non-zero entries (since a most common value of $e_{i,1},\ldots,e_{i,n}$ must be represented by at least $n/3$ entries). Since at most $(3/2)C$ of the $\beta_j$ are non-zero, we conclude that $R$ contains at most $(2/3)n + (3/2)C$ entries on row $i$ (that is, of the form $(i,\cdot)$), which is at most $(3/4)n$ for an appropriate $N$. A similar property holds for columns.

\smallskip

We now show that if we sample a random permutation $\pi \in S_n$, then with probability $\Omega(1)$ it avoids $R$. We sample $\pi$ in two stages. The first stage consists of sampling $n/5$ carefully chosen entries of $\pi$, and in the second stage we sample the remaining $4n/5$ entries of $\pi$. We assume for simplicity that $n$ is divisible by~$10$.

The first stage consists of $n/5$ iterations, labelled by $t \in \{0,\ldots,n/5-1\}$, which alternate between row iterations (when $t$ is even) and column iterations (when $t$ is odd). At the $t$'th iteration, $t$ values of $\pi$ have been sampled so far. Let $I_t$ consist of those indices $i \in [n]$ on which $\pi(i)$ is undefined, let $J_t$ consist of those indices $j \in [n]$ on which $\pi^{-1}(j)$ is undefined, and let $R_t = R \cap (I_t \times J_t)$. If $t$ is even, let $i \notin I_t$ be a choice which maximizes the number of entries in the $i$'th row of $R_t$. Denote this number by $m_t$, and sample $\pi(i)$ uniformly among $J_t$. If $t$ is odd, let $j \notin J_t$ be a choice which maximizes the number of entries in the $j$'th column of $R_t$. Denote this number by $m_t$, and sample $\pi^{-1}(j)$ uniformly among $I_t$.

The first stage is \emph{successful} if the defined values of $\pi$ avoid $R$. This happens with probability
\[
 p = \prod_{t=0}^{n/5-1} \left(1 - \frac{m_t}{n-t}\right).
\]
Recall that $R$ contains at most $(3/4)n$ entries on each row and column. Thus $m_t \leq (3/4)n$, and so $m_t/(n-t) \leq (3/4)n/(4/5)n = 15/16$. Since $\log(1-x)/-x$ is increasing, it is not hard to check that $1 - m_t/(n-t) \geq e^{-3m_t/(n-t)} \geq e^{-4m_t/n}$, and consequently
\[
 p \geq e^{-\sum_{t=0}^{n/5-1} 4m_t/n} \geq e^{-16C},
\]
since the sum of $m_t$ is at most $|R| \leq 4Cn$.

\smallskip

Let $I' = I_{n/5}$, $J' = J_{n/5}$, and $R' = R_{n/5}$. Thus after the first stage, in order to complete the sampling process we need to sample $\pi(i)$ for each $i \in I'$, the allowed values being $J'$; and the newly sampled values must avoid $R'$, in which case we say that the second stage is \emph{successful}. Furthermore, $n' := |I'| = |J'| = 4n/5$. 

We claim that $R'$ contains very few entries on each row or column. Indeed, suppose that some row of $R'$ contains $s$ entries. This means that $m_t \geq s$ for the $n/10$ even $t$ in the range $0,\ldots,n/5-1$, since otherwise this row would have been chosen in one of the row iterations. Consequently, $|R| \geq (n/10)s$. Since $|R| \leq 4Cn$, we conclude that $s \leq 40C$. Similarly, every column of $R'$ contains at most $40C$ entries. Moreover, $|R'| \leq 40Cn'$.

If $|R'| \leq n'/2$ then we can lower-bound the probability that the second stage is successful using the union bound: denoting by $\pi'$ the part of $\pi$ left undefined after the first stage,
\[
 \Pr[\pi' \text{ hits } R'] \leq \sum_{(i,j) \in R} \Pr[\pi'(i) = j] \leq \frac{|R'|}{n'} \leq \frac{1}{2}.
\]
In order to lower-bound the probability that the second stage is successful in general (when $|R'| \geq n'/2$), we use truncated inclusion-exclusion (also known as the Bonferroni inequality). Let $d \in \NN$ be a parameter to be chosen later. Then
\[
 \Pr[\pi' \text{ hits } R'] \leq \sum_{k=0}^{2d} (-1)^k \sum_{\substack{\{(i_1,j_1),\ldots,(i_k,j_k)\} \subset R' \\ \text{all different}}} \Pr[\pi'(i_1) = j_1, \ldots, \pi'(i_k) = j_k].
\]
We stress that the sum is over \emph{unordered} $k$-tuples of entries of $R'$.
Let us say that a $k$-tuple of entries is \emph{consistent} if all of $i_1,\ldots,i_k$ are distinct and all of $j_1,\ldots,j_k$ are distinct. If a $k$-tuple is consistent then the probability that $\pi'(i_s) = j_s$ for all $s \in [k]$ is $1/n^{\prime\underline{k}} := 1/n'(n'-1)\cdots(n'-k+1)$, and otherwise the probability is~$0$. Hence if $N_k$ is the number of unordered consistent $k$-tuples of entries in $R'$,
\begin{equation} \label{eq:sporadic}
 \Pr[\pi' \text{ hits } R'] \leq \sum_{k=0}^{2d} (-1)^k \frac{N_k}{n^{\prime\underline{k}}}.
\end{equation}

Let $M_k$ be the number of \emph{ordered} consistent $k$-tuples of entries in $R'$. Then $N_k = M_k/k!$. We proceed to give upper and lower bounds on both $M_k$ and $n^{\prime\underline{k}}$. Clearly $n^{\prime\underline{k}} \leq n^{\prime k}$, and on the other hand,
\[
 n^{\underline{k}} = n^k \prod_{s=0}^{k-1} \left(1 - \frac{s}{n}\right) \geq \left(1 - \frac{k(k-1)}{2n'}\right) n^k.
\]
As for $M_k$, clearly $M_k \leq |R'|^k$. On the other hand, we can lower bound $M_k$ as follows. There are $|R'|$ choices for the first entry. Since each row and column of $R'$ contains at most $40C$ entries, there are at least $|R'|-80C$ choices for the second entry, $|R'|-160C$ choices for the third entry, and so on. In total,
\[
 M_k \geq \prod_{s=0}^{k-1} (|R'| - 40Cs) = |R'|^k \prod_{s=0}^{k-1} \left(1 - \frac{40Cs}{|R'|}\right) \geq \left(1 - \frac{O(Ck^2)}{|R'|}\right) |R'|^k \geq \left(1 - \frac{O(Ck^2)}{n'}\right) |R'|^k,
\]
using the assumption $|R'| \geq n'/2$.

Putting both estimates together, we get that for even $k$,
\[
 \frac{N_k}{n^{\prime \underline{k}}} \leq \left(1 - \frac{O(k^2)}{n}\right)^{-1} \frac{|R'|^k}{k! n^{\prime k}} \stackrel{(\ast)}\leq 
 \left(1 + \frac{O(k^2)}{n}\right) \frac{(|R'|/n')^k}{k!} \leq
 \frac{(|R'|/n')^k}{k!} + \frac{O(k^2 (40C)^k)}{k! n},
\]
where $(\ast)$ holds whenever $n$ is large enough as a function of $k$.
Similarly, for odd $k$,
\[
 \frac{N_k}{n^{\prime \underline{k}}} \geq
 \left(1 - \frac{O(Ck^2)}{n}\right) \frac{(|R'|/n')^k}{k!} \geq
 \frac{(|R'|/n')^k}{k!} - \frac{O(Ck^2 (40C)^k)}{k! n}.
\]

Substituting these estimates in~\eqref{eq:sporadic}, we obtain that if $n$ is large enough as a function of $d$ then
\[
 \Pr[\pi' \text{ hits } R'] \leq \sum_{k=0}^{2d} (-1)^k \frac{(|R'|/n')^k}{k!} + O\left(\frac{1}{n}\right) \sum_{k=0}^{2d} \frac{(1+C)k^2 (40C)^k}{k!}.
\]
We can bound the second term by $O_C(1/n)$ by extending the sum to infinity and noting that the series converges. As for the first term, it is a truncation of the Taylor series of $e^{-x}$ at $x = |R'|/n'$. Since the Taylor series is an alternating sum, we can bound
\[
 e^{-|R'|/n'} \geq \sum_{k=0}^{2d} (-1)^k \frac{(|R'|/n')^k}{k!} - \frac{(|R'|/n')^{2d+1}}{(2d+1)!}.
\]
Altogether, using $1/2 \leq |R'|/n' \leq 40C$ we obtain
\[
 \Pr[\pi' \text{ hits } R'] \leq e^{-1/2} + \frac{(40C)^{2d+1}}{(2d+1)!} + O_C\left(\frac{1}{n}\right).
\]

Let $e^{-1/2} = 1 - \delta$.
We choose $d$ so that the second term is at most $\delta/3$, and we choose $N$ (the lower bound on $n$) so that $n$ is large enough as a function of $d$ for the estimates above to hold, and furthermore the third term is at most $\delta/3$. We conclude that $\pi'$ avoids $R'$ with probability at least $\delta/3$, and so the second stage is successful with probability at least $\min(1/2,\delta/3) = \delta/3$. Altogether, $\pi$ avoids $R$ with probability at least $e^{-16C}(\delta/3)$.

\subsection{Sum of cosets} \label{sec:sum-of-cosets}

Applying \Cref{lem:L2-sparse} and \Cref{lem:sporadic} in sequence, we approximate $f$ with a function $g = r + \sum_{i,j} r_{i,j} x_{i,j}$, such that with constant probability, a random permutation avoids all non-zero $r_{i,j}$'s. This implies that $g = r$ with constant probability, and so $r \in \{0,1\}$ (for small enough $\epsilon$). In order to say something similar about the coefficients $r_{i,j}$, we need a similar hitting property that holds relative to a coset $(i,j)$.

\begin{lemma} \label{lem:sporadic-relative}
Let $R \subseteq [n]^2$, and suppose that $|R| \leq Cn$ and that with probability at least $\gamma > 0$, a random permutation $\pi$ avoids $R$, that is $(i,\pi(i)) \notin R$ for all $i \in [n]$. If $n \geq 8C/\gamma + 2$ then for all $i,j \in [n]$,
\[
 \Pr_{\pi \in S_n}[\pi \text{ avoids } R \setminus \{(i,j)\} \mid \pi(i) = j] \geq \gamma/2.
\]

Similarly, if $n \geq 32C/\gamma + 3$ then for all $i_1,j_1,i_2,j_2 \in [n]$ such that $i_1 \neq i_2$ and $j_1 \neq j_2$,
\[
 \Pr_{\pi \in S_n}[\pi \text{ avoids } R \setminus \{(i_1,j_1),(i_2,j_2)\} \mid \pi(i_1) = j_1 \text{ and } \pi(i_2) = j_2] \geq \gamma/4.
\]
\end{lemma}
\begin{proof}
Let $\pi'$ be a uniformly random permutation. We define a permutation $\pi$ such that $\pi(i) = j$ as follows. If $\pi'(i) = j$, then we simply define $\pi = \pi'$. Otherwise, let $\pi'(i) = j'$ and $\pi'(i') = j$. Then $\pi$ is obtained from $\pi'$ by setting $\pi(i) = j$ and $\pi(i') = j'$. By symmetry, $\pi$ is a uniformly random permutation such that $\pi(i) = j$, and furthermore, in the second case, $(i',j')$ is chosen uniformly over all pairs such that $i' \neq i$ and $j' \neq j$.

If the permutation $\pi$ hits $R \setminus \{(i,j)\}$ then either $\pi'$ hits $R$ or $(i',j') \in R$. The first event happens with probability at most $1 - \gamma$, and a union bound shows that the second event happens with probability at most $|R'|/(n-1)^2 \leq Cn/(n-1)^2 \leq 4C/n \leq \gamma/2$. Hence $\pi$ avoids $R \setminus \{(i,j)\}$ with probability at least $\gamma - \gamma/2 \geq \gamma/2$. This concludes the proof of the first part of the lemma.

\smallskip

To prove the second part of the lemma, note first that since $n \geq 8C/\gamma + 2$,
\[
 \Pr_{\pi \in S_n}[\pi \text{ avoids } R \setminus \{(i_1,j_1)\} \mid \pi(i_1) = j_1] \geq \gamma/2.
\]
Since $n \geq 2$, we have $|R \setminus \{(i_1,j_1)\}| \leq Cn \leq 2C(n-1)$. Thinking of the coset $(i_1,j_1)$ as a copy of $S_{n-1}$, we apply the lemma again, with $n \gets n-1$, $R \gets R \cap \overline{\{i_1\}} \times \overline{\{j_1\}}$, $C \gets 2C$, and $\gamma \gets \gamma/2$, to deduce the second part of the lemma.
\end{proof}

Using this property, we can uncover the structure of functions satisfying the premises of \Cref{lem:sporadic}. We state the following lemma in terms of $g$ rather than $f$ to facilitate its reuse in \Cref{sec:L0}.

Recall that an \emph{$\epsilon$-disjoint family of cosets} $\cC \subseteq [n]^2$ consists of $O(n)$ many cosets such that the number of pairs of cosets in $\cC$ which are not disjoint is $O(\epsilon n^2)$. In the ``direct'' part of \Cref{lem:sum-of-cosets}, the hidden big~O constants depends only on~$C$, and in the ``converse'' part, they can be arbitrary.

\begin{lemma} \label{lem:sum-of-cosets}
Let $g\colon S_n \to \ZZ$ be a function satisfying the premises of \Cref{lem:sporadic}. There exist constants $N \in \NN$ and $\epsilon_0 > 0$, possibly depending on $C$, such that the following holds whenever $n \geq N$ and $\epsilon \leq \epsilon_0$.

There is an $\epsilon$-disjoint family of cosets $\cC$, and a choice $G \in \{g,1-g\}$, such that the function $h = \sum_{(i,j) \in \cC} x_{i,j}$ satisfies $\Pr[G \neq h] = O(\epsilon)$ and $\E[(G-h)^2] = O(\epsilon)$.

Conversely, if $\cC$ is an $\epsilon$-disjoint family of cosets then $h = \sum_{(i,j) \in \cC} x_{i,j}$ satisfies $\Pr[h \notin \{0,1\}] = O(\epsilon)$ and $\E[\dist(h,\{0,1\})^2] = O(\epsilon)$.
\end{lemma}
\begin{proof}
Let $b,N_1,\gamma$ be the constants promised by \Cref{lem:sporadic}, let $N_2 = 32C/\gamma+3$, and choose $N = \max(N_1,N_2)$.
Since $n \geq N_1$, \Cref{lem:sporadic} applies, and gives us a representation
\[
 g = r + \sum_{i=1}^n \sum_{j=1}^n r_{i,j} x_{i,j}
\]
with the stated properties. In particular, $\Pr[g = r] \geq \gamma$, and so if $\epsilon < \gamma$, we conclude that $r \in \{0,1\}$. If $r = 0$ then we take $G = g$, and otherwise we take $G = 1-g$. In the sequel, we assume for simplicity that $r = 0$; an identical argument works when $r = 1$.

Since $n \geq N_2$, we can apply \Cref{lem:sporadic-relative} to $R = \{(i,j) : r_{i,j} \neq 0\}$. Applying the lemma to any $(i,j) \in R$, we obtain
\[
 \Pr_{\pi \in S_n}[\pi(i) = j \text{ and } \pi \text{ avoids } R \setminus \{(i,j)\}] \geq \frac{\gamma/2}{n}.
\]
When this event happens, $g = r + r_{i,j} = r_{i,j}$. The events for different $(i,j) \in R$ are disjoint, and so if we denote $B = \{(i,j) \in R : r_{i,j} \neq 1\}$ then
\[
 \epsilon \geq \sum_{(i,j) \in B} \Pr_{\pi \in S_n}[\pi(i) = j \text{ and } \pi \text{ avoids } R \setminus \{(i,j)\}] \geq \frac{\gamma/2}{n} |B|,
\]
implying that $|B| \leq (2\epsilon/\gamma) n = O(\epsilon n)$.

We define $\cC = R \setminus B$. Since $\Pr[\pi(i) = j] = 1/n$ for each $(i,j) \in B$, the union bound shows that $\Pr[g \neq h] = O(\epsilon)$. Moreover,
\[
 \E[(g-h)^2] \leq \frac{1}{n} \sum_{(i,j) \in B} r_{i,j}^2 + \frac{1}{n(n-1)} \sum_{(i_1,j_1),(i_2,j_2) \in B} |r_{i_1,j_1} r_{i_2,j_2}| \leq \frac{b^2 |B|}{n} + \frac{b^2 |B|^2}{n(n-1)} = O(\epsilon).
\]

To complete the proof of the first part of the lemma, it remains to show that $\cC$ is $\epsilon$-disjoint. Clearly $|\cC| \leq |R| \leq Cn$. For every unordered pair $(i_1,j_1),(i_2,j_2) \in \cC$ of distinct non-disjoint cosets, since $n \geq N_2$ we can apply \Cref{lem:sporadic-relative} to obtain
\[
 \Pr_{\pi \in S_n}[\pi(i_1) = j_1 \text{ and } \pi(i_2) = j_2 \text{ and } \pi \text{ avoids } \cC \setminus \{(i_1,j_1),(i_2,j_2)\}] \geq \frac{\gamma/4}{n(n-1)}.
\]
When this event happens, $g = 2$. These events are disjoint for different unordered pairs, and so if we denote by $P$ the number of such unordered pairs then
\[
 \Pr[h = 2] \geq \frac{\gamma/4}{n(n-1)} P.
\]
On the other hand, $\Pr[h = 2] \leq \Pr[g = 2] + \Pr[g \neq h] = O(\epsilon)$, and so $P \leq O(\epsilon n^2)$. This shows that $\cC$ is $\epsilon$-disjoint.

\smallskip

Let us now turn to the converse part. Suppose that $\cC$ is an $\epsilon$-disjoint family of cosets, and let $h = \sum_{(i,j) \in \cC} x_{i,j}$. Again denoting by $P$ the number of unordered pairs of non-disjoint cosets in $\cC$, the union bound shows that
\[
 \Pr[h \notin \{0,1\}] = \Pr[h \geq 2] \leq \frac{P}{n(n-1)} = O(\epsilon).
\]

In order to bound $\E[\dist(h, \{0,1\})^2]$, we first note that if $z \in \NN$ then $\dist(z, \{0,1\})^2 = \Theta(z(z-1))$. Hence it suffices to bound $\E[h(h-1)]$:
\[
 \E[h(h-1)] = \E[h^2] - \E[h] = \frac{|\cC|}{n} + \frac{2P}{n(n-1)} - \frac{|\cC|}{n} = O(\epsilon). \qedhere
\]
\end{proof}

\Cref{lem:sum-of-cosets} forms the bulk of the proof of \Cref{thm:main-1,thm:main-2}.

\subsection{Constant or dictator} \label{sec:constant-or-dictator}

In order to prove \Cref{thm:main-1a,thm:main-2a}, we will analyze the function $h$ constructed in \Cref{lem:sum-of-cosets}.

Recall that a \emph{dictator} is a function of the form $\sum_{(i,j) \in \cC} x_{i,j}$ where the cosets in $\cC$ are disjoint, that is, all of them are on the same row, or all of them are on the same column.

\begin{lemma} \label{lem:constant-or-dictator}
Let $\cC \subseteq [n]^2$ be an $\epsilon$-disjoint family of cosets, and let $h = \sum_{(i,j) \in \cC} x_{i,j}$.

For every $\delta \geq \sqrt{\epsilon}$, either $\Pr[h \neq 0] = O(\delta)$ and $\E[h^2] = O(\delta)$, or there exists a dictator $H$ such that $\Pr[h \neq H] = O(\epsilon/\delta)$ and $\E[(h - H)^2] = O(\epsilon/\delta)$.
\end{lemma}
\begin{proof}
Let us first notice that
\[
 \E[h^2] \leq \frac{|\cC|}{n} + \frac{|\cC| (|\cC| - 1)}{n(n-1)} = O(1).
\]
This completes the proof when $\delta \geq 1$, since $\Pr[h \neq 0] \leq 1 \leq \delta$ and $\E[h^2] = O(1) = O(\delta)$. From now on, we assume that $\delta < 1$. The lemma also trivially holds when $n = 1$, so we can assume that $n \geq 2$.


By definition of $\epsilon$-disjoint family, there exists a constant $L>0$ such that the number of unordered pairs of non-disjoint cosets in $\cC$ is at most $L\epsilon n^2$.
Since $\delta \geq \sqrt{\epsilon}$, if $\delta < 1/\sqrt{L}n$ then $L\epsilon n^2 \leq L\delta^2 n^2 < 1$, and so $h$ is a dictator, completing the proof. From now on, we assume that $\delta \geq 1/\sqrt{L}n$.

\smallskip

Let $K > 0$	be a constant to be determined. We say that a line (row or column) of $\cC$ is \emph{heavy} if it contains at least $K\delta n$ entries.

Suppose first that there are no heavy lines. Then every row and column contains at most $K\delta n$ cosets in $\cC$. Each coset is thus disjoint from at most $2K\delta n$ many cosets in $\cC$, and so the number of unordered pairs of non-disjoint cosets in $\cC$ is at least $|\cC| (|\cC| - 2K\delta n)/2$. If $|\cC| \geq 4K\delta n$ then $|\cC| (|\cC| - 2K\delta n)/2 \geq |\cC|^2/4 \geq 4K^2 \delta^2 n \geq 4K^2\epsilon n$, which contradicts the upper bound $L \epsilon n^2$ for $K > \sqrt{L}/2$. Hence $|\cC| \leq 4K\delta n$.

We claim that in this case, $\Pr[h \neq 0] = O(\delta)$ and $\E[h^2] = O(\delta)$. Indeed, the union bound shows that $\Pr[h \neq 0] \leq 4K\delta  = O(\delta)$, and
\[
 \E[h^2] \leq \frac{|\cC|}{n} + \frac{L \epsilon n^2}{n(n-1)} = O(\delta + \epsilon) = O(\delta),
\]
since $\epsilon \leq \delta^2 < \delta$.

Suppose next that $\cC$ contains some heavy line, say $\cL$. Every entry in $\cC \setminus \cL$ is disjoint from at most one entry of $\cL$, and so the number of unordered pairs of non-disjoint cosets in $\cC$ is at least $|\cC \setminus \cL| (|\cL| - 1)$. Since $\delta \geq 1/\sqrt{L} n$, we have
\[
 |\cL| - 1 \geq K\delta n \left(1-\frac{1}{K\delta n}\right) \geq K\delta n \left(1-\frac{1}{K/\sqrt{L}}\right),
\]
which is at least $K\delta n/2$ for $K \geq 2\sqrt{L}$. In total, the number of unordered pairs of non-disjoint cosets is at least $|\cC \setminus \cL| K\delta n/2$.
Since this number is at most $L\epsilon n^2$, we conclude that $|\cC \setminus \cL| \leq 2(L/K)(\epsilon/\delta)n$.

The required dictator is $H = \sum_{(i,j) \in \cL} x_{i,j}$. The union bound shows that $\Pr[h \neq H] \leq 2(L/K)(\epsilon/\delta) = O(\epsilon/\delta)$. Moreover,
\[
 \E[(h - H)^2] \leq \frac{|\cC \setminus \cL|}{n} + \frac{L \epsilon n^2}{n(n-1)} = O(\epsilon/\delta + \epsilon) = O(\epsilon/\delta),
\]
since $\delta < 1$.
\end{proof}

\subsection{Main theorems} \label{sec:L2-main}

Let us recall the characterization of Boolean linear functions due to Ellis, Friedgut and Pilpel~\cite{EFP11} (for the definition of dictator, see \Cref{sec:constant-or-dictator}).

\begin{theorem} \label{thm:EFP}
 If $f\colon S_n \to \{0,1\}$ is linear then $f$ is a dictator.
\end{theorem}

We can now prove \Cref{thm:main-1,thm:main-2,thm:main-1a,thm:main-2a}. We start with \Cref{thm:main-2,thm:main-2a}.
Recall that $f$ is $\epsilon$-close to Boolean if $\E[\dist(f,\{0,1\})^2] \leq \epsilon$. 

\maintwo*

\begin{proof}
 The converse follows directly from \Cref{lem:sum-of-cosets}, so it suffices to prove the first part of the theorem.

 The arguments below require that $n \geq N$ and $\epsilon \leq \epsilon_0$, for some constants $N \in \NN$ and $\epsilon_0 > 0$ that derive from the various lemmas. We first dispense with the case $\epsilon > \epsilon_0$.

 If $\epsilon > \epsilon_0$ then we can choose $\cC = \emptyset$, and so $g = 0$, since $\E[f^2] = O(\E[\dist(f,\{0,1\})^2] + \E[\round(f,\{0,1\})^2]) = O(\epsilon + 1) = O(\epsilon)$. From now on, we assume that $\epsilon \leq \epsilon_0$.

 \Cref{lem:L2-sparse} constructs a function $g$ of a certain form such that $\E[(f-g)^2] = O(\epsilon)$ and $\Pr[g \notin \{0,1\}] = O(\epsilon)$.
 
 We can now dispense with the case $n < N$. We can assume that $\epsilon_0 \leq c/N!$, where $c>0$ is chosen so that $\Pr[g \notin \{0,1\}] < 1/N!$. Consequently, $g$ is Boolean, and so a dictator by \Cref{thm:EFP}. This completes the proof of the theorem when $n < N$. From now on, we assume that $n \geq N$.
 
 \Cref{lem:sporadic} shows that $g$ can be written in an equivalent form, which allows application of \Cref{lem:sum-of-cosets}. According to the latter, there is an $\epsilon$-disjoint family of cosets $\cC$ and a choice $G \in \{g,1-g\}$ such that $h = \sum_{(i,j) \in \cC} x_{i,j}$ satisfies $\E[(G - h)^2] = O(\epsilon)$. If $G = g$ then $\E[(f - h)^2] = O(\E[(f - g)^2] + \E[(g - h)]^2) = O(\epsilon)$, completing the proof, so suppose that $G = 1-g$.
 
 When $G = 1-g$, we have $\E[((1-f)-h)^2] = O(\E[((1-f)-(1-g))^2] + \E[((1-g)-h)^2]) = O(\epsilon)$, and so $(\E[1-f] - \E[h])^2 = O(\epsilon)$. This implies that $\E[h] \geq \E[1-f] - O(\epsilon) \geq 1/2 - O(\epsilon)$, and so $\E[h] \geq 1/3$ for an appropriate choice of $\epsilon_0$. Since $h \in \ZZ$, also $\E[h^2] \geq \E[h] \geq 1/3$. Applying \Cref{lem:constant-or-dictator} for an appropriate constant $\delta$, we obtain that some dictator $H$ satisfies $\E[(h - H)^2] = O(\epsilon)$, and so $\E[(f-(1-H))^2] = \E[((1-f)-H)^2] = O(\E[((1-f)-h)^2] + \E[(h-H)^2]) = O(\epsilon)$. This completes the proof, since $1 - H$ is also a dictator.
\end{proof}

\maintwoa*

\begin{proof}
 We can assume that $\E[f] \leq 1/2$, since otherwise we can repeat the argument with $1-f$, using the fact that if $H$ is a dictator then so is $1-H$. Moreover, clearly $\delta \leq 1-\delta$ and so $\delta \leq 1/2$.

 As in the proof of \Cref{thm:main-2}, we can assume that $n \geq N$ and $\epsilon \leq \epsilon_0$. The proof constructs a function $g$ satisfying $\E[(f-g)^2] = O(\epsilon)$, and shows that there is an $\epsilon$-disjoint family of cosets $\cC$ and a choice $G \in \{g,1-g\}$ such that $h = \sum_{(i,j) \in \cC} x_{i,j}$ satisfies $\E[(G - h)^2] = O(\epsilon)$.
 
 Suppose first that $G = g$. Then $\E[(f - h)^2] = O(\epsilon)$, and so $(\E[f] - \E[h])^2 = O(\epsilon)$, implying that $\E[h] \geq \E[f] - O(\sqrt{\epsilon}) \geq \delta/2$, for an appropriate choice of $K$. Since $h \in \ZZ$, also $\E[h^2] \geq \E[h] \geq \delta/2$. Applying \Cref{lem:constant-or-dictator} with $\delta \gets K'\delta$ for an appropriate choice of $K'$, we obtain that $\E[(h - H)^2] = O(\epsilon/\delta)$ for some dictator $H$, and so $\E[(f - H)^2] = O(\epsilon/\delta + \epsilon) = O(\epsilon/\delta)$, since $\delta \leq 1/2$.
 
 If $G = 1 - g$ then $\E[((1-f) - h)^2] = O(\epsilon)$ and so $(\E[1-f] - \E[h])^2 = O(\epsilon)$, implying that $\E[h] \geq \E[1-f] - O(\sqrt{\epsilon}) \geq 1/2 - O(\sqrt{\epsilon}) \geq 1/3$, for an appropriate choice of $\epsilon_0$. Applying the argument of the case $g = G$ with $\delta \gets 2/3$, we obtain that $\E[(1-f) - H)^2] = O(\epsilon) = O(\epsilon/\delta)$ for some dictator $H$. This completes the proof, since $\E[(f - (1-H))^2] = O(\epsilon/\delta)$, and $1-H$ is also a dictator.
\end{proof}

Finally, we derive \Cref{thm:main-1,thm:main-1a}. Recall that $f$ is $\epsilon$-close to linear if $\E[(f^{>1})^2] \leq \epsilon$.

\mainone*

\begin{proof}
 We start with the first part of the theorem. 
 Let $F = f^{\leq 1}$. By assumption, $\E[(f - F)^2] \leq \epsilon$, and so $\E[\dist(f,\{0,1\})^2] \leq \E[(f - F)^2] \leq \epsilon$. Moreover, $\E[F] = \E[f] \leq 1/2$. Applying \Cref{thm:main-2}, there is an $\epsilon$-disjoint family of cosets $\cC$ such that $G = \sum_{(i,j) \in \cC} x_{i,j}$ satisfies $\E[(F - G)^2] = O(\epsilon)$. Therefore $\E[(f - G)^2] = O(\E[(f - F)^2] + \E[(F - G)^2]) = O(\epsilon)$. Since $f$ is Boolean and $G \in \ZZ$, this implies that $\Pr[f \neq G] = O(\epsilon)$.
 If we define $g = \max_{(i,j) \in \cC} x_{i,j}$ then $f = G$ implies $f = g$, and so $\Pr[f \neq g] \leq \Pr[f \neq G] = O(\epsilon)$.
 
 We move on to the converse part. Suppose that $\cC$ is an $\epsilon$-disjoint family of cosets, let $G = \sum_{(i,j) \in \cC} x_{i,j}$, and let $g = \max_{(i,j) \in \cC} x_{i,j}$. According to \Cref{thm:main-2}, $\E[\dist(G,\{0,1\})^2] = O(\epsilon)$. Since $g = \round(G, \{0,1\})$, this shows that $\E[(g - G)^2] = O(\epsilon)$, and so $g$ is $O(\epsilon)$-close to linear.
\end{proof}

\mainonea*

\begin{proof}
 Let $F = f^{\leq 1}$, so that $\E[\dist(f,\{0,1\})^2] \leq \E[(f-F)^2] \leq \epsilon$ and $\E[F] = \E[f]$. Applying \Cref{thm:main-2a}, we obtain that $\E[(F-g)^2] = O(\epsilon/\delta)$ for some dictator $g$. Thus $\E[(f-g)^2] = O(\E[(f-F)^2] + \E[(F-g)^2]) = O(\epsilon + \epsilon/\delta) = O(\epsilon/\delta)$ (since $\delta \leq 1/2$). Since $f,g$ are both Boolean, this implies that $\Pr[f \neq g] = O(\epsilon/\delta)$.
\end{proof}

\section{Approximation in \texorpdfstring{$L_0$}{L0}} \label{sec:L0}

In this section we prove \Cref{thm:main-3,thm:main-3a}. Most of the effort is dedicated to proving an $L_0$~analog of \Cref{lem:L2-sparse}, which we do in \Cref{sec:L0-sparse} using the $L_0$~analog of the FKN theorem on the cube, proved in \Cref{sec:L0-FKN}, and a new take on the reduction to the cube, which appears in \Cref{sec:L0-reduction}. We then derive \Cref{thm:main-3,thm:main-3a} in \Cref{sec:L0-main}.

\subsection{FKN theorem on the cube} \label{sec:L0-FKN}

We start by proving an $L_0$~analog of the Friedgut--Kalai--Naor theorem on the Boolean cube.

\begin{lemma} \label{lem:L0-FKN}
 Let $f\colon \{0,1\}^n \to \RR$ be given by
\[
 f(x_1,\ldots,x_n) = c + \sum_{i=1}^n c_i x_i.
\]
 If $\Pr[f \notin \{0,1\}] < 1/4$ then either $f \in \{0, 1\}$ or $f \in \{x_i, 1-x_i\}$ for some $i \in [n]$.
 
 In particular, $c_1,\ldots,c_n \in \{0,\pm1\}$, and at most one of $c_1,\ldots,c_n$ is non-zero.
\end{lemma}
\begin{proof}
 We start by showing that $c_1,\ldots,c_n \in \{0,\pm 1\}$. If $c_i \notin \{0,\pm 1\}$, then for each assignment to all variables other than $x_i$, there is at least one assignment to $x_i$ such that $f(x_1,\ldots,x_n) \notin \{0,1\}$, and so $\Pr[f \notin \{0,1\}] \geq 1/2$, contradicting the assumption.
 
 Similarly, we claim that at most one of $c_1,\ldots,c_n$ can be non-zero. If $c_i,c_j \neq 0$, then $c_i x_i + c_j x_j$ attains at least three different values as $x_i,x_j \in \{0,1\}$.
 Therefore for each assignment to all variables other than $x_i,x_j$, there is at least one assignment to $x_i,x_j$ such that $f(x_1,\ldots,x_n) \notin \{0,1\}$, and so $\Pr[f \notin \{0,1\}] \geq 1/4$, contradicting the assumption.
 
 We are left with three options: $f = c$, $f = c + x_i$, $f = c - x_i$. In the first case, necessarily $c \in \{0,1\}$. In the second case, necessarily $c = 0$, since otherwise $\Pr[f \notin \{0,1\}] \geq 1/2$. Similarly, in the third case, necessarily $c = 1$.
\end{proof}

\Cref{cor:FKN} has the useful feature that if $N \geq 2$ of $c_1,\ldots,c_n$ are far from zero (more than $1/2$ in absolute value) then $\epsilon = \Omega(N - 1)$. In contrast, even if all of $c_1,\ldots,c_n$ are non-zero then we cannot improve on $\Pr[f \notin \{0,1\}] \leq 1$ (the bound~$1$ should be compared to the bound $\Omega(N - 1)$). This makes it harder to approximate $f$ by a sparse sum in the $L_0$~setting.

\subsection{Reduction to the cube} \label{sec:L0-reduction}

Our proof will require a slightly different take on the reduction to the cube described in \Cref{sec:L2-reduction}, together with an argument along the lines of \Cref{lem:sporadic-relative}.

Recall the definition of $S_{\va,\vb}$. Given two permutations $\va,\vb \in S_n$, the subset $S_{\va,\vb} \subset S_n$ consists of all permutations which map $\{\va_{2t-1},\va_{2t}\}$ to $\{\vb_{2t-1},\vb_{2t}\}$ for all $t \in \lfloor n/2 \rfloor$.

We call a set of the form $\{a_1,a_2\} \times \{b_1,b_2\}$, where $a_1,a_2,b_1,b_2 \in [n]$, $a_1 \neq a_2$, and $b_1 \neq b_2$, a \emph{square}. A \emph{singleton} is a set of the form $\{a\} \times \{b\}$, where $a,b \in [n]$. Two squares $\{a_1,a_2\} \times \{b_1,b_2\}, \{a_3,a_4\} \times \{b_3,b_4\}$ are \emph{compatible} if $\{a_1,a_2\} \cap \{a_3,a_4\} = \emptyset$ and $\{b_1,b_2\} \cap \{b_3,b_4\} = \emptyset$. Similarly, the square $\{a_1,a_2\} \times \{b_1,b_2\}$ is compatible with the singleton $\{a_3\} \times \{b_3\}$ if $a_3 \neq a_1,a_2$ and $b_3 \neq b_1,b_2$.

We associate permutations $\va,\vb \in S_n$ with the set $\bigl\{ \{\va_{2t-1},\va_{2t}\} \times \{\vb_{2t-1}, \vb_{2t}\} : t \in [\lfloor n/2 \rfloor] \bigr\}$ of $\lfloor n/2 \rfloor$ compatible squares, accompanied by the compatible singleton $\{\va_n\} \times \{\vb_n\}$ when $n$ is odd. This set of $\lceil n/2 \rceil$ compatible squares and (zero or one) singletons completely specifies $S_{\va,\vb}$. We call such a set a \emph{square system} (although when $n$ is odd, it also contains a singleton). If we choose $\va,\vb$ at random, then the resulting square system is a \emph{(uniformly) random square system}.

A permutation $\pi$ \emph{hits} a square $\{a_1,a_2\} \times \{b_1,b_2\}$ if $\pi(\{a_1,a_2\}) = \{b_1,b_2\}$. If $\pi$ does not hit a square $\{a_1,a_2\} \times \{b_1,b_2\}$, then whenever $\pi \in S_{\va,\vb}$, the corresponding square system does not contain the square $\{a_1,a_2\} \times \{b_1,b_2\}$.

We denote by $S_\Sigma$ the set of permutations hitting all squares in the square system $\Sigma$. If $\Sigma$ is a random square system and we choose a random element of $S_\Sigma$, then the result is a random element of $S_n$.

\smallskip

We move on to the analog of \Cref{lem:sporadic-relative}. We are given a set of bad squares. In our applications this set will consist of squares $\{a_1,a_2\} \times \{b_1,b_2\}$ such that $c_{a_1,b_1} + c_{a_2,b_2} - c_{a_1,b_2} - c_{a_2,b_1}  \notin \{0,\pm1\}$ (in one application) or $c_{a_1,b_1} + c_{a_2,b_2} - c_{a_1,b_2} - c_{a_2,b_1} \neq 0$ (in another), where $f = \sum_{i,j} c_{i,j} x_{i,j}$.

We are given that a random square system typically contains a small number of bad squares (in our applications, zero or one). We want to say that this still holds even conditioned on the square system containing one or two given squares.

\begin{lemma} \label{lem:L0-relative}
For $n \geq 2$, let $R \subseteq \binom{[n]}{2} \times \binom{[n]}{2}$ be a collection of squares of density $\rho = |R|/\binom{n}{2}^2$, and suppose that for some $C \in \NN$,
\[
 \Pr_\Sigma[|\Sigma \cap R| > C] = \delta,
\]
where $\Sigma$ is a random square system.

For any square $\sigma$,
\[
 \Pr_\Sigma[|\Sigma \cap (R \setminus \{\sigma\})| > C \mid \sigma \in \Sigma] \leq \delta + O(\rho).
\]

For any two compatible squares $\sigma_1,\sigma_2$,
\[
 \Pr_\Sigma[|\Sigma \cap (R \setminus \{\sigma_1,\sigma_2\})| > C \mid \sigma_1,\sigma_2 \in \Sigma] \leq \delta + O(\rho).
\]
\end{lemma}
\begin{proof}
Let $\Sigma'$ be a random square system. We wish to modify it so that it contains $\sigma = \{a_1,a_2\} \times \{b_1,b_2\}$. We do this in a number of steps.

First, suppose that $n$ is odd. If the singleton of $\Sigma'$ is $\{a_1\} \times \{B\}$ then we switch $a_1$ with a random element in $[n] \setminus \{a_1,a_2\}$. We handle singletons of the forms $\{a_2\} \times \{B\}, \{A\} \times \{b_1\}, \{A\} \times \{b_2\}$ in an analogous way (some singletons could require two fixing steps).

At this point, regardless of the parity of $n$, we know that $a_1,a_2,b_1,b_2$ all participate in squares. If $a_1,a_2$ are not part of the same square, then they are part of two different squares $\{a_1,A_2\} \times \{B_1,B_2\}$ and $\{A_1,a_2\} \times \{C_1,C_2\}$. We replace these squares with the squares $\{a_1,a_2\} \times \{B_1,B_2\}$ and $\{A_1,A_2\} \times \{C_1,C_2\}$. We act similarly if $b_1,b_2$ are not part of the same square.

Finally, if $\{a_1,a_2\}$ and $\{b_1,b_2\}$ are not part of the same square, then they are part of two different squares $\{a_1,a_2\} \times \{B_1,B_2\}$ and $\{A_1,A_2\} \times \{b_1,b_2\}$. We replace these squares with the squares $\{a_1,a_2\} \times \{b_1,b_2\}$ and $\{A_1,A_2\} \times \{B_1,B_2\}$.

Denote the resulting random square system by $\Sigma$. This is a uniformly random square system containing the square $\sigma$. By construction, $|\Sigma \setminus \Sigma'| = O(1)$. By symmetry, if we choose $\Sigma'$ uniformly at random and a random square in $\Sigma \setminus \Sigma'$, we either get $\sigma$ or a uniformly random square compatible with $\sigma$. The number of squares compatible with $\sigma$ is $\binom{n-2}{2}^2$, which is $\Omega(\binom{n}{2}^2)$ assuming $n \geq 4$ (otherwise the lemma trivially holds). Therefore a union bound shows that the probability that $\Sigma \setminus \Sigma'$ contains any square in $R \setminus \{\sigma\}$ is $O(\rho)$. By assumption, $|\Sigma' \cap (R \setminus \{\sigma\})| > C$ with probability at most $\delta$, and so overall, the probability that $|\Sigma \cap (R \setminus \{\sigma\})| > C$ is at most $\delta + O(\rho)$.
This proves the first part of the lemma.

\smallskip

Now suppose that instead of a single square $\sigma$ we are given two compatible squares $\sigma_1,\sigma_2$. We can assume that $n \geq 4$, since otherwise the lemma trivially holds. Applying the first part of the lemma to $\sigma_1$, we obtain
\[
 \Pr_\Sigma[|\Sigma \cap (R \setminus \{\sigma_1\})| > C \mid \sigma_1 \in \Sigma] \leq \delta + O(\rho).
\]
If we remove $\sigma_1$ from $\Sigma$, we obtain a random square system in a copy of $S_{n-2}$. The density of $R \setminus \{\sigma_1\}$ inside this copy is $\rho' = |R \setminus \{\sigma_1\}|/\binom{n-2}{2}^2 = O(\rho)$. Applying the first part of the lemma to $\sigma_2$, we obtain
\[
 \Pr_\Sigma[|\Sigma \cap (R \setminus \{\sigma_1,\sigma_2\})| > C \mid \sigma_1,\sigma_2 \in \Sigma] \leq \delta + O(\rho). \qedhere
\]
\end{proof}

\subsection{Sparse representation} \label{sec:L0-sparse}

Our goal in this subsection is to prove the following analog of \Cref{lem:L2-sparse}.

\begin{lemma} \label{lem:L0-sparse}
 Let $f\colon S_n \to \RR$ be a linear function satisfying $\Pr[f \notin \{0,1\}] = \epsilon$. Then there is an integer $e$ and integers $e_{i,j} \in \{0,\pm1\}$ such that the function
 \[
  g = e + \sum_{i=1}^n \sum_{j=1}^n e_{i,j} x_{i,j}
 \]
 satisfies $\Pr[f \neq g] = O(\epsilon)$, and so $\Pr[g \notin \{0,1\}] = O(\epsilon)$. Moreover, only $O(n)$ many of the $e_{i,j}$ are non-zero.
\end{lemma}

In the proof, we can assume that $\epsilon < \epsilon_0$ for an appropriate constant $\epsilon_0 > 0$. Indeed, if $\epsilon \geq \epsilon_0$ then we can satisfy the lemma by choosing $g = 0$.

Similarly, we can assume that $n > N$ for an appropriate constant $N \in \NN$. Indeed, if $n \leq N$ then either $\epsilon \geq 1/N!$, in which case we are done by choosing $\epsilon_0$ small enough, or $\epsilon < 1/N!$, in which case $f$ is Boolean. In the latter case, we can take $g = f$ and apply \Cref{thm:EFP} to represent $f$ in the required form.

\smallskip

Since $f$ is a linear function, we can write
\[
 f = \sum_{i,j} c_{i,j} x_{i,j}.
\]
For $i_1 ,i_2,j_1,j_2 \in [n]$, define $d_{i_1,i_2,j_1,j_2} = c_{i_1,j_1} + c_{i_2,j_2} - c_{i_1,j_2} - c_{i_2,j_1}$. If we switch $i_1,i_2$ or $j_1,j_2$, then this expression changes sign. This ensures that the following sets are well-defined: $R_0$ is the set of squares $\{i_1,i_2\} \times \{j_1,j_2\}$ such that $d_{i_1,i_2,j_1,j_2} \neq 0$, and $R_1$ is the set of squares $\{i_1,i_2\} \times \{j_1,j_2\}$ such that $d_{i_1,i_2,j_1,j_2} \notin \{0,\pm1\}$. Note that if $i_1 = i_2$ or $j_1 = j_2$ then $d_{i_1,i_2,j_1,j_2} = 0$.

Let $\Sigma$ be a random square system, let $f_\Sigma = f|_{S_\Sigma}$, and let $\epsilon_\Sigma = \Pr[f_\Sigma \notin \{0,1\}]$. Thus $\E[\epsilon_\Sigma] = \epsilon$, and so $\Pr[\epsilon_\Sigma \geq 1/4] \leq 4\epsilon$.

Whenever $\epsilon_\Sigma < 1/4$, \Cref{lem:L0-FKN} shows that $\Sigma \cap R_1 = \emptyset$, and so $\Pr[|\Sigma \cap R_1| > 0] \leq 4\epsilon$. Moreover, if we choose $\Sigma$ at random, then a random square of $\Sigma$ belongs to $R_1$ with probability at most $4\epsilon$. Since a random square of $\Sigma$ is just a random square, this shows that $\rho_1 = |R_1|/\binom{n}{2}^2 \leq 4\epsilon$.
Applying \Cref{lem:L0-relative} with $C = 0$ for any $\sigma \in R_1$, we deduce that
\[
 \Pr_\Sigma[|\Sigma \cap R_1| = \{\sigma\}] =
 \frac{\lfloor n/2 \rfloor}{\binom{n}{2}^2} \left(1 - \Pr_\Sigma[|\Sigma \cap (R_1 \setminus \{\sigma\})| > 0 \mid \sigma \in \Sigma]\right) =
 \Omega\left(\frac{1}{n^3}\right) (1 - O(\epsilon)) = \Omega\left(\frac{1}{n^3}\right),
\]
assuming $\epsilon_0$ is small enough.
These events are disjoint for different $\sigma \in R_1$, and so
\[
 \Pr_\Sigma[\Sigma \cap R_1 \neq \emptyset] = \Omega\left(\frac{|R_1|}{n^3}\right).
\]
Since this probability is at most $4\epsilon$, we conclude that $|R_1| = O(\epsilon n^3)$, and so a random square belongs to $R_1$ with probability $O(\epsilon/n)$.

\smallskip

\Cref{lem:L0-FKN} shows that if $\epsilon_\Sigma < 1/4$ then not only $\Sigma \cap R_1 = \emptyset$, but moreover $|\Sigma \cap R_0| \leq 1$. Therefore if we choose a random square system $\Sigma$ and a random square in $\Sigma$, then it belongs to $R_0$ with probability at most $\rho_0 = 4\epsilon + 1/\lfloor n/2 \rfloor = O(\epsilon + 1/n)$.
Applying \Cref{lem:L0-relative} with $C = 1$ for any compatible $\sigma_1,\sigma_2 \in R_0$, we deduce that
\begin{multline*}
 \Pr_\Sigma[|\Sigma \cap (R_0 \setminus \{\sigma_1,\sigma_2\})| \leq 1] =
 \frac{\lfloor n/2 \rfloor (\lfloor n/2 \rfloor - 1)}{\binom{n}{2}^2 \binom{n-2}{2}^2} \left(1 - \Pr_\Sigma[|\Sigma \cap (R_0 \setminus \{\sigma_1,\sigma_2\})| > 1 \mid \sigma_1,\sigma_2 \in \Sigma] \right) = \\
 \Omega\left(\frac{1}{n^6}\right) \bigl(1 - O(\epsilon + 1/n)\bigr) = \Omega\left(\frac{1}{n^6}\right),
\end{multline*}
assuming $\epsilon_0$ is small enough and $N$ is large enough. If any of these events happens then $2 \leq |\Sigma \cap R_0| \leq 3$, and so each $\Sigma$ can belong to at most three of these events. Therefore, if we denote by $P$ the number of ordered pairs of compatible squares in $R_0$ then
\[
 \Pr[|\Sigma \cap R_0| \geq 2] = \Omega\left(\frac{P}{n^6}\right).
\]
Since this probability is at most $O(\epsilon)$, we conclude that $P = O(\epsilon n^6)$.
On the other hand, since a given square is incompatible with at most $4n^3$ other squares,
\[
 P \geq |R_0| (|R_0| - 4n^3).
\]
Hence either $|R_0| \leq 8n^3$ or $P \geq |R_0|^2/2$, which implies that $|R_0| = O(\sqrt{\epsilon} n^3) = O(n^3)$. In both cases, $|R_0| = O(n^3)$, and so a random square belongs to $R_0$ with probability $O(1/n)$.

\smallskip

Summarizing our work so far, we have shown that a random square belongs to $R_1$ with probability $O(\epsilon/n)$ and to $R_0$ with probability $O(1/n)$. We can sample a random square $\{i_1,i_2\} \times \{j_1,j_2\}$ by first sampling $i_1,j_1 \in [n]$ and then sampling $i_2,j_2 \in [n]$ subject to $i_2 \neq i_1$ and $j_2 \neq j_1$. Therefore there exist $i_1,j_1 \in [n]$ such that the $d_{i,j} := d_{i_1,i,j_1,j}$ satisfy
\begin{align*}
 &\Pr_{\substack{i \neq i_1 \\ j \neq j_1}}[d_{i,j} \notin \{0,\pm1\}] = O(\epsilon/n), &
 &\Pr_{\substack{i \neq i_1 \\ j \neq j_1}}[d_{i,j} \neq 0] = O(1/n).
\end{align*}
When $i = i_1$ or $j = j_1$, we have $d_{i,j} = 0$, and so the bounds above hold even if we sample $i,j \in [n]$ in an unrestricted fashion.

As in the proof of \Cref{lem:L2-sparse}, we have
\[
 f = \sum_{i=1}^n \sum_{j=1}^n (d_{i,j} - c_{i_1,j_1} + c_{i_1,j} + c_{i,j_1}) x_{i,j} =
 -nc_{i_1,j_1} + \sum_{j=1}^n c_{i_1,j} + \sum_{i=1}^n c_{i,j_1} + \sum_{i=1}^n \sum_{j=1}^n d_{i,j} x_{i,j},
\]
and so for an appropriate $e$,
\[
 f = e + \sum_{i=1}^n \sum_{j=1}^n d_{i,j} x_{i,j}.
\]

We let $e_{i,j} = d_{i,j}$ if $d_{i,j} \in \{0,\pm 1\}$, and $e_{i,j} = 0$ otherwise. This defines a function
\[
 g = e + \sum_{i=1}^n \sum_{j=1}^n e_{i,j} x_{i,j}
\]
which differs from $f$ only when the input permutation $\pi$ satisfies $d_{i,\pi(i)} \notin \{0,\pm 1\}$ for some $i \in [n]$. Applying the union bound, we obtain
\[
 \Pr[f \neq g] \leq \frac{\bigl|\bigl\{ (i,j) \in [n]^2 : d_{i,j} \notin \{0,\pm 1\} \bigr\}\bigr|}{n} = O(\epsilon).
\]
This implies that $\Pr[g \notin \{0,1\}] \leq \Pr[f \notin \{0,1\}] + \Pr[f \neq g] = O(\epsilon)$, which is strictly smaller than~$1$ for small enough $\epsilon_0$. Therefore $g(\pi) \in \{0,1\}$ for some $\pi \in S_n$. Since $e_{i,j} \in \{0,\pm 1\}$, this implies that $e$ is an integer. Finally, by construction the number of non-zero $e_{i,j}$ is at most the number of non-zero $d_{i,j}$, which is $O(n)$.

\subsection{Main theorems} \label{sec:L0-main}

We can now prove \Cref{thm:main-3,thm:main-3a}. The proofs are similar to the ones in \Cref{sec:L2-main}.

Recall that $f$ is $\epsilon$-close to Boolean in $L_0$ if $\Pr[f \notin \{0,1\}] \leq \epsilon$.

\mainthree*

\begin{proof}
 The converse follows directly from \Cref{lem:sum-of-cosets}, so it suffices to prove the first part of the theorem.

 The arguments below require that $n \geq N$ and $\epsilon \leq \epsilon_0$, for some constants $N \in \NN$ and $\epsilon_0 > 0$ that originate from the various lemmas. If $\epsilon > \epsilon_0$	then we can take $g = 0$, since $\Pr[f \neq g] \leq 1 = O(\epsilon)$. If $n < N$ then either $\epsilon \geq 1/N!$, in which case we are done as before, or $\epsilon < 1/N!$, in which case $f$ is Boolean, and so \Cref{thm:EFP} completes the proof. From now on, we assume that $\epsilon \leq \epsilon_0$ and $n \geq N$.
 
 \Cref{lem:L0-sparse} constructs a function $g$ of a certain form satisfying $\Pr[f \neq g] = O(\epsilon)$. \Cref{lem:sporadic} shows that $g$ can be written in an equivalent form, which allows application of \Cref{lem:sum-of-cosets}. According to the latter, there is an $\epsilon$-disjoint family of cosets $\cC$ and a choice $G \in \{g,1-g\}$ such that $h = \sum_{(i,j) \in \cC} x_{i,j}$ satisfies $\Pr[G \neq h] = O(\epsilon)$. If $G = g$ then $\Pr[f \neq h] \leq \Pr[f \neq g] + \Pr[g \neq h] = O(\epsilon)$, completing the proof, so suppose that $G = 1-g$.
 
 When $G = 1-g$, we have $\Pr[1-f \neq h] \leq \Pr[1-f \neq 1-g] + \Pr[1-g \neq h] = O(\epsilon)$, and so $\Pr[h \neq 0] \geq \Pr[f \neq 1] - O(\epsilon) \geq 1/2 - O(\epsilon)$. For an appropriate choice of $\epsilon_0$, this implies that $\Pr[h \neq 0] \geq 1/3$. Applying \Cref{lem:constant-or-dictator} for an appropriate constant $\delta$, we obtain that $\Pr[h \neq H] = O(\epsilon)$ for some dictator $H$. Thus $\Pr[f \neq 1-H] \leq \Pr[f \neq 1-h] + \Pr[h \neq H] = O(\epsilon)$. Since $1 - H$ is also a dictator, this completes the proof. 
\end{proof}

\mainthreea*

\begin{proof}
We can assume that $\Pr[f = 1] \leq 1/2$	, since otherwise we can repeat the argument with $1-f$, using the fact that if $H$ is a dictator then so is $1-H$. Moreover, clearly $\delta \leq 1/2$.

As in the proof of \Cref{thm:main-3}, we can assume that $n \geq N$ and $\epsilon \leq \epsilon_0$. The proof constructs a function $g$ satisfying $\Pr[f \neq g] = O(\epsilon)$, and shows that there is an $\epsilon$-disjoint family of cosets $\cC$ and a choice $G \in \{g,1-g\}$ such that $h = \sum_{(i,j) \in \cC} x_{i,j}$ satisfies $\Pr[G \neq h] = O(\epsilon)$.

Suppose first that $G = g$. Then $\Pr[f \neq h] \leq \Pr[f \neq g] + \Pr[g \neq h] = O(\epsilon)$, and so $\Pr[h \neq 0] \geq \Pr[f = 1] - O(\epsilon) \geq \delta - O(\epsilon) \geq \delta/2$, for an appropriate choice of $K$. Applying \Cref{lem:constant-or-dictator} with $\delta \gets K'\delta$ for an appropriate choice of $K'$, we obtain that $\Pr[h \neq H] = O(\epsilon/\delta)$ for some dictator $H$, and so $\Pr[f \neq H] \leq \Pr[f \neq h] + \Pr[h \neq H] = O(\epsilon/\delta + \epsilon) = O(\epsilon/\delta)$, since $\delta \leq 1/2$.

If $G = 1 - g$ then $\Pr[1-f \neq h] = O(\epsilon)$, and so $\Pr[h \neq 0] \geq \Pr[f \neq 1] - O(\epsilon) \geq 1/2 - \epsilon \geq 1/3$, for an appropriate choice of $\epsilon_0$. Applying the argument of the case $g = G$ with $\delta \gets 2/3$, we obtain that $\Pr[1-f \neq H] = O(\epsilon) = O(\epsilon/\delta)$ for some dictator $H$. This completes the proof, since $\Pr[f \neq 1-H] = O(\epsilon)$, and $1-H$ is also a dictator.
\end{proof}

\section{Approximation in \texorpdfstring{$L_\infty$}{Linf}} \label{sec:Linf}

In this section we prove \Cref{thm:main-4}. Although not necessary for proving \Cref{thm:main-4}, we start this section with an $L_\infty$ version of the FKN theorem on the Boolean cube, which is stated and proved in \Cref{sec:Linf-FKN}. We then prove an analog of \Cref{lem:L2-sparse,lem:L0-sparse} in \Cref{sec:Linf-sparse}, and an analog of \Cref{lem:sporadic} in \Cref{sec:Linf-sporadic}. Finally, we prove \Cref{thm:main-4} in \Cref{sec:Linf-main}.

\subsection{FKN theorem on the cube} \label{sec:Linf-FKN}

For reference, here is an $L_\infty$ analog of the Friedgut--Kalai--Naor theorem on the Boolean cube.

\begin{lemma} \label{lem:Linf-FKN}
Let $f\colon \{0,1\}^n \to \RR$ be given	 by
\[
 f(x_1,\ldots,x_n) = c + \sum_{i=1}^n c_i x_i.
\]
If $\dist(f(x),\{0,1\}) \leq \epsilon$ for all $x \in \{0,1\}^n$ and $\epsilon < 1/4$ then the function $g(x) = \round(f(x), \{0,1\})$ satisfies $g \in \{0,1,x_1,\ldots,x_n,1-x_1,\ldots,1-x_n\}$.
\end{lemma}
\begin{proof}
Since $f(0,\ldots,0) = c$, we see that $\dist(c,\{0,1\}) \leq \epsilon$. Suppose that $|c| \leq \epsilon$; otherwise we can consider $1 - f$. Let $e^{(i)} \in \{0,1\}^n$ have a unique $1$-coordinate at position~$i$. Since $f(e^{(i)}) - f(0,\ldots,0) = c_i$, we see that $\dist(c_i,\{0,1\}) \leq 2\epsilon$. Since $f(e^{(i)} + e^{(j)}) - f(0,\ldots,0) = c_i + c_j$, also $\dist(c_i + c_j,\{0,1\}) \leq 2\epsilon$.

Let $d_i = \round(c_i,\{0,1\})$. If $d_i = d_j = 1$ then  $c_i + c_j \geq 2 - 2\epsilon > 1 + 2\epsilon$, which is impossible. Therefore at most one $d_i$ is non-zero. If $d_j = 0$ then $|c_j| \leq 2\epsilon$, and so if $x,y$ differ only in the $j$'th coordinate then $|f(x) - f(y)| = |c_j| \leq 2\epsilon$, and so $g(x) = g(y)$, since otherwise $|f(x) - f(y) \geq 1-2\epsilon > 2\epsilon$. Consequently, $g$ depends on at most one coordinate.
\end{proof}

\subsection{Sparse representation} \label{sec:Linf-sparse}

Our goal in this subsection is to prove the following analog of \Cref{lem:L2-sparse,lem:L0-sparse}. In the statement of the lemma, $x$ is \emph{$\epsilon$-close} to $y$ if $|x-y| \leq \epsilon$.

\begin{lemma} \label{lem:Linf-sparse}
Let $f\colon S_n \to \RR$ be a linear function satisfying $\dist(f(\pi), \{0,1\}) \leq \epsilon$ for all $\pi \in S_n$, where $\epsilon < 1/6$. We can write
\[
 f = d + \sum_{i=1}^n \sum_{j=1}^n d_{i,j} x_{i,j}
\]
in such a way that each $d_{i,j}$ is $2\epsilon$-close to $\{0,\pm 1\}$, and furthermore at most $O(n)$ of the $d_{i,j}$ are $2\epsilon$-close to~$\{\pm1\}$.
\end{lemma}

If $n = 1$ then the lemma trivially holds, so we can assume that $n \geq 2$.

For $i_1,i_2,j_1,j_2 \in [n]$, define $d_{i_1,i_2,j_1,j_2} = c_{i_1,j_1} + c_{i_2,j_2} - c_{i_1,j_2} - c_{i_2,j_1}$.
Let $\va,\vb \in S_n$ be any two permutations, and consider $S_{\va,\vb}$. For every $t \in [\lfloor n/2 \rfloor]$, we can find two permutations $\alpha,\beta \in S_{\va,\vb}$ such that $f(\alpha) - f(\beta) = d_{\va_{2t-1},\va_{2t},\vb_{2t-1},\vb_{2t}}$. Since $\va,\vb$ are arbitrary, this shows that $\dist(d_{i_1,i_2,j_1,j_2}, \{0,\pm1\}) \leq 2\epsilon$ for all $i_1,i_2,j_1,j_2 \in [n]$ (this trivially holds if $i_1 = i_2$ or $j_1 = j_2$).

Similarly, for every distinct $s,t \in [\lfloor n/2 \rfloor]$ we can find two permutations $\alpha,\beta \in S_{\va,\vb}$ such that $f(\alpha) - f(\beta) = d_{\va_{2s-1},\va_{2s},\vb_{2s-1},\vb_{2s}} \pm d_{\va_{2t-1},\va_{2t},\vb_{2t-1},\vb_{2t}}$, for our choice of sign. This shows that $\round(d_{\va_{2t-1},\va_{2t},\vb_{2t-1},\vb_{2t}},\{0,\pm1\}) \neq 0$ for at most one $t \in [\lfloor n/2 \rfloor]$ (here we are using $\epsilon < 1/6$). 
Therefore a random rectangle $\{i_1,i_2\} \times \{j_1,j_2\}$ satisfies $d_{i_1,i_2,j_1,j_2} \neq 0$ with probability at most $1/\lfloor n/2 \rfloor = O(1/n)$.

We can fix values $i_1,j_1$ so that the number of $d_{i,j} := d_{i_1,i,j_1,j}$ which are not $2\epsilon$-close to~$0$ is at most
\[
 (n-1)^2 \Pr_{\substack{i \neq i_1 \\ j \neq j_1}}[d_{i_1,i,j_1,j} \neq 0] = O(n).
\]

As in the proofs of \Cref{lem:L2-sparse,lem:L0-sparse}, we have
\[
 f = \sum_{i=1}^n \sum_{j=1}^n (d_{i,j} - c_{i_1,j_1} + c_{i_1,j} + c_{i,j_1}) x_{i,j} =
 -nc_{i_1,j_1} + \sum_{j=1}^n c_{i_1,j} + \sum_{i=1}^n c_{i,j_1} + \sum_{i=1}^n \sum_{j=1}^n d_{i,j} x_{i,j},
\]
and so for an appropriate $d$,
\[
 f = d + \sum_{i=1}^n \sum_{j=1}^n d_{i,j} x_{i,j}.
\]

\subsection{Sporadic representation} \label{sec:Linf-sporadic}

Next, we prove an analog of \Cref{lem:sporadic}.

\begin{lemma} \label{lem:Linf-sporadic}
Let $f\colon S_n \to \RR$ be a linear function satisfying $\dist(f(\pi), \{0,1\}) \leq \epsilon$ for all $\pi \in S_n$, where $\epsilon < 1/4$ and $n \geq N$ for an appropriate constant $N \in \NN$. We can write
\[
 f = e + \sum_{i=1}^n \sum_{j=1}^n e_{i,j} x_{i,j}
\]
in such a way that each $e_{i,j}$ is $2\epsilon$-close to an integer; at most $O(n)$ of the $e_{i,j}$ are $2\epsilon$-close to non-zero integers; and each row or column of the $n \times n$ matrix formed by the $e_{i,j}$ contains at least $n/4$ entries which are $2\epsilon$-close to~$0$.
\end{lemma}

The proof uses the first part of the argument in the proof of \Cref{lem:sporadic}. Our starting point is the representation
\[
 f = d + \sum_{i=1}^n \sum_{j=1}^n d_{i,j} x_{i,j}
\]
promised by \Cref{lem:Linf-sparse}, in which at most $Cn$ of the $d_{i,j}$ are $2\epsilon$-close to~$0$.

Let $D_{i,j} = \round(d_{i,j}, \{0,\pm 1\})$,
let $\alpha_i \in \{-1,0,1\}$ be a most common value of $D_{i,1},\ldots,D_{i,n}$, let $\beta_j \in \{-1,0,1\}$ be the most common value of $D_{1,j},\ldots,D_{n,j}$, and define $e_{i,j} = d_{i,j} - \alpha_i - \beta_j$, so that
\[
 g = d + \sum_{i=1}^n \sum_{j=1}^n (e_{i,j} + \alpha_i + \beta_j) x_{i,j} = d + \sum_{i=1}^n \alpha_i + \sum_{j=1}^n \beta_j + \sum_{i=1}^n \sum_{j=1}^n e_{i,j} x_{i,j}.
\]

We claim that most of the coefficients $\alpha_i,\beta_j$ are equal to zero. Indeed, if $\alpha_i \neq 0$ then among $d_{i,1},\ldots,d_{i,n}$, at least $(2/3)n$ are not $2\epsilon$-close to~$0$. Therefore at most $Cn/(2/3)n = (3/2)C$ of the $\alpha_i$ are non-zero. Similarly, at most $(3/2)C$ of the $\beta_j$ are non-zero.
If $E_{i,j} := \round(e_{i,j}, \ZZ) \neq 0$ then at least one of $D_{i,j},\alpha_i,\beta_j$ is non-zero. This shows that the number of non-zero $E_{i,j}$ is at most $Cn + (3/2)Cn + (3/2)Cn = 4Cn$.

Given $i \in [n]$, notice that by construction, $D_{i,1} - \alpha_i, \ldots, D_{i,n} - \alpha_i$ contains at most $(2/3)n$ many non-zero entries. Since at most $(3/2)C$ of the $\beta_j$ are non-zero, we conclude that the $n \times n$ matrix formed by the $E_{i,j}$ contains at most $(2/3)n + (3/2)C$ non-zero entries on row $i$, which is at most $(3/4)n$ for an appropriate $N$. A similar property holds for columns.

\subsection{Main theorem} \label{sec:Linf-main}

We are now ready to uncover the structure of $f$.

\begin{lemma} \label{lem:Linf-sum-of-cosets}
Let $f\colon S_n \to \RR$ be a linear function satisfying $\dist(f(\pi), \{0,1\}) \leq \epsilon$ for all $\pi \in S_n$, where $\epsilon < 1/40$ and $n \geq N$ for an appropriate constant $N \in \NN$. Then $g(\pi) = \round(f(\pi), \{0,1\})$ is a dictator.
\end{lemma}
\begin{proof}
Apply \Cref{lem:Linf-sporadic} to get a representation
\[
 f = e + \sum_{i=1}^n \sum_{j=1}^n e_{i,j} x_{i,j}.
\]
with the properties stated by the lemma, and let $E_{i,j} = \round(e_{i,j}, \ZZ)$. According to the lemma, $|e_{i,j} - E_{i,j}| \leq 2\epsilon$, and at most $O(n)$ of the $E_{i,j}$ are non-zero.

For a rectangle $\{i_1,i_2\} \times \{j_1,j_2\}$, let $e_{i_1,i_2,j_1,j_2} = e_{i_1,j_1} + e_{i_2,j_2} - e_{i_1,j_2} - e_{i_2,j_1}$, and let $E_{i_1,i_2,j_1,j_2} = E_{i_1,j_1} + E_{i_2,j_2} - E_{i_1,j_2} - E_{i_2,j_1}$. We can find two permutations $\alpha,\beta$ such that $f(\alpha) - f(\beta) = e_{i_1,i_2,j_1,j_2}$, and so $\dist(e_{i_1,i_2,j_1,j_2}, \{0,\pm 1\}) \leq 2\epsilon$. Since $|e_{i_1,i_2,j_1,j_2} - E_{i_1,i_2,j_1,j_2}| \leq 8\epsilon$, this shows that $\dist(E_{i_1,i_2,j_1,j_2}, \{0,\pm 1\}) \leq 10\epsilon$, and so $E_{i_1,i_2,j_1,j_2} \in \{0,\pm 1\}$.

If $\{i_1,i_2\} \times \{j_1,j_2\}$ and $\{i_3,i_4\} \times \{j_3,j_4\}$ are two compatible squares, then we can find two permutations $\alpha,\beta$ such that $f(\alpha) - f(\beta) = e_{i_1,i_2,j_1,j_2} \pm e_{i_3,i_4,j_3,j_4}$, for a sign of our choice. This implies that $\dist(E_{i_1,i_2,j_1,j_2} \pm E_{i_3,i_4,j_3,j_4}, \{0,\pm 1\}) \leq 18\epsilon$, and so $E_{i_1,i_2,j_1,j_2},E_{i_3,i_4,j_3,j_4}$ cannot both be non-zero.

\smallskip

Suppose that $E_{i_1,j_1} \neq 0$. There are at least $n/4$ many $i_2$ such that $E_{i_2,j_1} = 0$, and at least $n/4$ many $j_2$ such that $E_{i_2,j_2} = 0$. Since at most $O(n)$ of the $E_{i,j}$ are non-zero, for large enough $N$ there must be a choice of $i_2$ and $j_2$ out of these $(n/4)^2$ options such that furthermore $E_{i_2,j_2} = 0$, and so $E_{i_1,i_2,j_1,j_2} = E_{i_1,j_1}$. This shows that $E_{i,j} \in \{0,\pm 1\}$ for all $i,j \in [n]$.

We now run the same argument on pairs. Suppose that $E_{i_1,j_1},E_{i_3,j_3} \neq 0$, where $i_1 \neq i_3$ and $j_1 \neq j_3$. There are $\Omega(n^2)$ many choices for $i_2,i_4$ such that $i_1,i_2,i_3,i_4$ are all distinct and $E_{i_2,j_1} = E_{i_4,j_3} = 0$ (assuming $N \geq 4$). Similarly, there are $\Omega(n^2)$ many choices for $j_2,j_4$ such that $j_1,j_2,j_3,j_4$ are all distinct and $E_{i_1,j_2} = E_{i_3,j_4} = 0$. Since at most $O(n)$ of the $E_{i,j}$ are non-zero, at most $O(n^2)$ out of these $\Omega(n^4)$ choices satisfy $E_{i_2,j_2} \neq 0$ or $E_{i_4,j_4} \neq 0$, and so for large enough $N$, there must be some choice of $i_2,i_4,j_2,j_4$ such that the squares $\{i_1,i_2\} \times \{j_1,j_2\}$ and $\{i_3,i_4\} \times \{j_3,j_4\}$ are compatible, $E_{i_1,i_2,j_1,j_2} = E_{i_1,j_1}$, and $E_{i_3,i_4,j_3,j_4} = E_{i_3,j_3}$. This shows that all non-zero $E_{i,j}$ must line on the same row or column.

Suppose that all non-zero $E_{i,j}$ lie on row $I$ (the other case is similar). If $E_{I,j_1} = 1$ and $E_{I,j_2} = -1$ then for any $i \in I$ we have $E_{I,i,j_1,j_2} = 2$, which impossible. By possibly replacing $f$ with $1-f$, we can assume that $E_{I,j} \in \{0,1\}$ for all $j \in [n]$.

\smallskip

Recall that $g(\pi) = \round(f(\pi), \{0,1\})$.
If $\alpha,\beta \in S_n$ are two permutations differing on a transposition and satisfying $\alpha(I) = \beta(I)$ then $f(\alpha) - f(\beta) = e_{i_1,i_2,j_1,j_2}$ for some $i_1,i_2 \neq I$. Since $E_{i_1,i_2,j_1,j_2} = 0$, this implies that $|g(\alpha) - g(\beta)| \leq 10\epsilon$, and so $g(\alpha) = g(\beta)$. 

Assuming $N \geq 3$, any two permutations agreeing on the image of $I$ can be obtained from one another by applying a sequence of transpositions. Therefore $g(\pi)$ depends only on $\pi(I)$, and so it is a dictator.
\end{proof}

To deduce \Cref{thm:main-4}, it remains to handle the case in which $n$ is small.

\mainfour*

\begin{proof}
Let $N$ be the constant from \Cref{lem:Linf-sum-of-cosets}. If $n \geq N$ then we are done, so assume that $n < N$. Apply \Cref{lem:Linf-sparse} to obtain a representation
\[
 f = d + \sum_{i=1}^n \sum_{j=1}^n d_{i,j} x_{i,j}
\]
such that $\dist(d_{i,j}, \{0,\pm1\}) \leq 2\epsilon$ for all $i,j$. Denoting by $\id$ the identity permutation, we have
\[
 f(\id) = d + \sum_{i=1}^n d_{i,i},
\]
and so $f(\id)$ is $(2n+1)\epsilon$-close to an integer.

Let $D = \round(d,\ZZ)$ and $D_{i,j} = \round(d_{i,j}, \{0,\pm 1\})$. Then the function
\[
 h = D + \sum_{i=1}^n D_{i,j} x_{i,j}
\]
satisfies $|f(\pi) - h(\pi)| \leq (4n+1)\epsilon$ for all $\pi \in S_n$, and so $\dist(h(\pi),\{0,1\}) \leq (4n+2)\epsilon$ for all $\pi \in S_n$. Choosing $\epsilon_0 < 1/(4N+2)$, this guarantees that $h$ is Boolean, and so a dictator by \Cref{thm:EFP}. Choosing $\epsilon_0 < 1/(8N+2)$, we see that $h(\pi) = \round(f(\pi), \{0,1\}) = g(\pi)$ for all $\pi \in S_n$, completing the proof.
\end{proof}

\bibliographystyle{amsplain-url}
\bibliography{filmusbiblio}

\begin{dajauthors}
\begin{authorinfo}[yuval]
  Yuval Filmus\\
  The Henry and Marylin Taub Faculty of Computer Science\\
  Technion --- Israel Institute of Technology\\
  Haifa, Israel\\
  \href{mailto:yuvalfi@cs.technion.ac.il}{yuvalfi@cs.technion.ac.il} \\
  \url{https://yuvalfilmus.cs.technion.ac.il}
\end{authorinfo}
\end{dajauthors}
	
\end{document}